\documentclass{article}
\usepackage{ulem}

\usepackage{amssymb,amsmath,graphicx,amsthm,mathrsfs}
\usepackage[textsize=footnotesize,color=DarkGreen!40]{todonotes}

\usepackage[colorlinks=true]{hyperref}
\usepackage{pdfsync,color}
\newtheorem{theorem}{Theorem}
\newtheorem{lemma}{Lemma}
\newtheorem{proposition}{Proposition}
\theoremstyle{definition}

\theoremstyle{definition}\newtheorem{remark}{Remark}
\numberwithin{equation}{section}

\newcommand{\lc}
{\mathrel{\raise2pt\hbox{${\mathop<\limits_{\raise1pt\hbox
{\mbox{$\sim$}}}}$}}}

\newcommand{\gc}
{\mathrel{\raise2pt\hbox{${\mathop>\limits_{\raise1pt\hbox{\mbox{$\sim$}}}}$}}}

\newcommand{\ec}
{\mathrel{\raise2pt\hbox{${\mathop=\limits_{\raise1pt\hbox{\mbox{$\sim$}}}}$}}}

\def\bb{\begin{equation}} \def\ee{\end{equation}}

\def\beqn{\begin{eqnarray}}  \def\eqn{\end{eqnarray}}

\def\beqnx{\begin{eqnarray*}} \def\eqnx{\end{eqnarray*}}

\def\bn{\begin{enumerate}} \def\en{\end{enumerate}}

\def\bd{\begin{description}} \def\ed{\end{description}}

\def\nb{\nonumber}
\def\label{\label}
\def \ra {\rightarrow}
\def \o  {\overline}
\def \p  {\partial}
\def \l  {\lambda}
\def \d  {\displaystyle}
\def \f  {\frac}
\def \e  {\varepsilon}

\newcommand{\OCP}{\bf (OCP)}
\newcommand{\IOCPn}{\bf (IOCP)_n}
\newcommand{\SOCPn}{\bf (SOCP)_n}
\renewcommand{\leq}{\leqslant}
\renewcommand{\geq}{\geqslant}

\title{Impulse and sampled-data optimal control of heat equations, and error estimates}

\author{Emmanuel Tr\'{e}lat\thanks{Sorbonne Universit\'{e}s,
UPMC Univ Paris 06, CNRS UMR 7598, Laboratoire Jacques-Louis Lions, Institut Universitaire de France, F-75005, Paris, France. (emmanuel.trelat@upmc.fr).}
\and
Lijuan Wang\thanks{School of Mathematics and Statistics,
Computational Science Hubei Key Laboratory, Wuhan University, Wuhan
430072, China. (ljwang.math@whu.edu.cn).}
\and
Yubiao Zhang\thanks{
School of Mathematics and Statistics, Wuhan University, Wuhan
430072, China. (yubiao\b{ }zhang@whu.edu.cn).} }

\begin{document}

\date{}

\maketitle

\begin{abstract}
We consider the optimal control problem of minimizing some quadratic functional over all possible solutions of an internally controlled multi-dimensional heat equation with a periodic terminal state constraint.
This problem has a unique optimal solution, which can be characterized by an optimality system derived from the Pontryagin maximum principle.
We define two approximations of this optimal control problem. The first one is an impulse approximation, and consists of considering a system of linear heat equations with impulse control. The second one is obtained by the sample-and-hold procedure applied to the control, resulting into a sampled-data approximation of the controlled heat equation.
We prove that both problems have a unique optimal solution, and we establish precise error estimates for the optimal controls and optimal states of the initial problem with respect to its impulse and sampled-data approximations.
\end{abstract}

\medskip

\noindent\textbf{2010 Mathematics Subject Classifications.}
35K05, 
49J20, 
34A37, 
93C57 

\medskip

\noindent\textbf{Keywords.}
Heat equation, optimal control problem, impulse control, sampled-data control, error estimates.

\section{Introduction and main results}
\subsection{The context}
There is a vast literature on numerical approximations of optimal control problems settled for parabolic differential equations. In the linear quadratic regulator (LQR) problem, many results do exist concerning space semi-discretizations of the Riccati procedure. 
We refer to \cite{Banks-Ito, Banks-Kunisch, Gibson, KappelSalamon, LiuZheng, RosenWang} for general results showing convergence of the approximations of the Riccati operator, under assumptions of uniform exponential stabilizability, and of uniform boundedness of the sequence of approximate Riccati solutions.
In \cite{Banks-Ito, LasieckaTriggiani1, LiuZheng}, these sufficient conditions (and thus, the convergence result) are proved to hold true in the general parabolic case and for unbounded control operators.
Note that, in such LQR problems, the final point is not fixed. When there is a terminal constraint the situation is more intricate, because things may go badly when discretizing optimal control problems in infinite dimension, due to interferences with the mesh that may cause the divergence of the optimization procedure when the mesh size is going to zero. These interferences are stronger when the terminal constraint has infinite codimension, in spite of strong dissipativity properties of parabolic equations.
For the optimal control problem of minimizing the $L^2$ norm of the control (corresponding to the celebrated ``Hilbert Uniqueness Method"), one can find results on uniform exact controllability and/or observability of discretized control systems in \cite{Boyer, Boyer2, ErvedozaValein, ErvedozaZuazua, Labbe-Trelat, Zuazua-4} (see also references therein), for different discretization processes on different parabolic models. It can be noted that uniformity requires in general to add some appropriate viscosity terms in the numerical scheme.
Besides, when the convergence is ensured, it is important to be able to derive error estimates which are as sharp as possible, and we refer the reader to \cite{Falk, Gunzburger, Geveci, Hou, Lasie, Lasiecka, Liu, Tiba, WangWang, Wang} for situations where Galerkin finite element approximations are used.

In many cases impulse control is an interesting alternative, not only to usual discretization schemes, but also in order to deal with systems that cannot be acted on by means of continuous control inputs, as it often occurs in applications.
For example, relevant controls for acting on a population of bacteria should be impulsive, so that the density of the bactericide may change instantaneously; indeed continuous control would enhance drug resistance of bacteria.
For more discussions and examples about impulse control or impulse control problems in infinite dimension, we refer the readers to \cite{Bensoussan, Yong, Yang} and references therein.
It is also interesting to note that impulse control is as well an alternative to the well known concept of digital control, or sampled-data control, which is much used in the engineering community.

To the best of our knowledge, error estimates for impulse approximations or for sampled-data approximations of an optimal control problem settled with partial differential equations and with continuous control inputs have not been investigated.

In this paper, we consider the problem of deriving precise error estimates for impulse
approximations and for sampled-data approximations of a linear quadratic optimal control problem settled for an internally controlled linear homogeneous heat equation with periodic terminal state constraint. The latter periodicity requirement is motivated by the fact that steady solutions and periodic solutions are of particular interest when considering parabolic differential equations.

\subsection{Definition of the optimal control problems}\label{sec12}
Let $N\geq 1$ be an integer, let $\Omega\subset \mathbb{R}^N$ be a bounded open set having a $C^2$ boundary $\p \Omega$, let $\omega\subset \Omega$ be an open non-empty subset, and let $T>0$ and $y_d\in L^2(0,T;L^2(\Omega))$ be arbitrary.
Throughout the paper, the norm in $L^2(\Omega)$ is denoted by $\Vert\ \Vert$.

\paragraph{The optimal control problem $\OCP$.}
We consider the optimal control problem $\OCP$ of minimizing the functional
\begin{equation}\label{def_J}
J(y,u)=\f{1}{2}\int_0^T\Vert y-y_d\Vert^2\,\mathrm dt+\f{1}{2}\int_0^T\Vert u\Vert^2\,\mathrm dt
\end{equation}
over all $(y,u)\in L^2(0,T;H^2(\Omega)\cap H_0^1(\Omega)) \cap H^1(0,T;L^2(\Omega))   \times L^2(0,T;L^2(\Omega))$ such that
\begin{eqnarray}\label{eq_pb}
\left\{\begin{array}{lll}
\p_t y-\triangle y=\chi_\omega u&\mbox{in}&\Omega \times
(0,T),\\
y=0&\mbox{on}&\p \Omega \times (0,T),\\
y(0)=y(T)&\mbox{in}&\Omega .
\end{array}\right.\end{eqnarray}
Here, $\chi_\omega$ designates the characteristic function of $\omega$, and $y$ and $u$ are functions of $(t,x)$.

We have the following facts:
\begin{itemize}
\item Given any $u\in L^2(0,T;L^2(\Omega))$, there exists a unique solution $y\in L^2(0,T;H^2(\Omega)\cap H_0^1(\Omega)) \cap H^1(0,T;L^2(\Omega))$ of \eqref{eq_pb}.
\item The problem $\OCP$ has a unique optimal solution $(y^*,u^*)$.
\item By the Pontryagin maximum principle (in short, PMP; see \cite{LiYong}), which is here a necessary and sufficient condition for optimality because the problem is linear quadratic, this minimizer is characterized by the existence of $p^* \in H^1(0,T;L^2(\Omega))\cap L^2(0,T;H^2(\Omega)\cap H_0^1(\Omega))$ such that
\begin{equation}
\left\{
\begin{array}{lll}
\p_t y^*-\triangle y^*=\chi_\omega u^*&\mbox{in}&\Omega \times
(0,T),\\
y^*=0&\mbox{on}&\p \Omega \times (0,T),\\
y^*(0)=y^*(T)&\mbox{in}&\Omega ,
\end{array}\right.\label{Maxi-1}
\end{equation}
\begin{equation}
\left\{
\begin{array}{lll}
\p_t p^*+\triangle p^*=y^*-y_d&\mbox{in}&\Omega \times
(0,T),\\
p^*=0&\mbox{on}&\p \Omega \times (0,T),\\
p^*(0)=p^*(T)&\mbox{in}&\Omega ,
\end{array}\right.\label{Maxi-2}
\end{equation}
and
\begin{equation}\label{Maxi-3}
u^*=\chi_\omega p^*\;\;\mbox{in}\;\;\Omega\times (0,T).
\end{equation}
\end{itemize}
These three claims are easy to establish, but for completeness they are proved in Section \ref{sec21}.

We are next going to design an approximating impulse optimal control problem $\IOCPn$, and an approximating sampled-data optimal control problem $\SOCPn$, for a linear heat equation with periodic terminal state constraint. Both problems have as well a unique solution, to which we will apply the PMP. We will then establish error estimates between the optimal solutions of $\OCP$ and, respectively, $\IOCPn$ and $\SOCPn$.

\paragraph{The approximating impulse optimal control problem $\IOCPn$.}
Let us define the approximating impulse optimal control problem $\IOCPn$, for $n\geq 2$.
We set
\begin{align*}
& h_n= T/n, \qquad \tau_i= i\,h_n, \qquad i=0, 1, \ldots, n,\\
& X = \prod_{i=1}^n X_i,\qquad  X_i= L^2(\tau_{i-1},\tau_i;H_0^1(\Omega))\cap H^1(\tau_{i-1},\tau_i;H^{-1}(\Omega)),\qquad i=1, 2, \ldots, n,
\end{align*}
and we define the functional $J_n:X\times (L^2(\Omega))^{n-1}\rightarrow [0,+\infty)$ by
\begin{equation}\label{Penalty}
J_n(Y_n,U_n)= \f{1}{2}\left(\sum_{i=1}^n \int_{\tau_{i-1}}^{\tau_i}\|y_{i,n}-y_d\|^2\,\mathrm dt+
\f{1}{h_n}\sum_{i=2}^n \|u_{i-1,n}\|^2\right),
\end{equation}
for $Y_n=(y_{1,n},y_{2,n},\ldots,y_{n,n})\in X$ and
$U_n=(u_{1,n},u_{2,n},\ldots,u_{n-1,n})\in (L^2(\Omega))^{n-1}$.
Here and throughout, $\|\cdot\|$ designates the norm in $L^2(\Omega)$. Accordingly, the inner product is denoted by $\langle\cdot,\cdot\rangle$.

We consider the impulse optimal control problem $\IOCPn$, consisting of minimizing the functional $J_n$ over all possible $(Y_n,U_n)\in X\times (L^2(\Omega))^{n-1}$ such that
\begin{equation}\label{impulse}\left\{\begin{array}{lll}
\p_t y_{i,n}-\triangle y_{i,n}=0&\mbox{in}&\Omega\times
(\tau_{i-1},\tau_i),\quad 1\leq i\leq n,\\
y_{i,n}=0&\mbox{on}&\p\Omega\times (\tau_{i-1},\tau_i),\quad 1\leq i\leq n,\\
y_{i,n}(\tau_{i-1})=y_{i-1,n}(\tau_{i-1})+\chi_\omega u_{i-1,n}&\mbox{in}&\Omega,\quad 2\leq i\leq n,\\
y_{1,n}(0)=y_{n,n}(T)&\mbox{in}&\Omega.
\end{array}\right.\end{equation}

\begin{proposition}\label{Principle}
For every $n\geq 2$, the optimal control problem $\IOCPn$ has a unique solution $(Y_n^*,U_n^*)$, with $Y_n^*=(y_{1,n}^*,y_{2,n}^*,\ldots,y_{n,n}^*)$ and $U_n^*=(u_{1,n}^*,u_{2,n}^*,\ldots,u_{n-1,n}^*)$.
The optimal solution $(Y_n^*,U_n^*)$ of $\IOCPn$ is characterized by the existence of
$p_n^*\in L^2(0,T;H^2(\Omega)\cap H_0^1(\Omega))\cap H^1(0,T;L^2(\Omega))$
such that
\begin{equation}\label{Prin-1}
\left\{\begin{array}{lll}
\p_t y_{i,n}^*-\triangle y_{i,n}^*=0&\mbox{in}&\Omega\times
(\tau_{i-1},\tau_i),\quad 1\leq i\leq n,\\
y_{i,n}^*=0&\mbox{on}&\p\Omega\times (\tau_{i-1},\tau_i),\quad 1\leq i\leq n,\\
y_{i,n}^*(\tau_{i-1})=y_{i-1,n}^*(\tau_{i-1})+\chi_\omega u_{i-1,n}^*&\mbox{in}&\Omega,\quad 2\leq i\leq n,\\
y_{1,n}^*(0)=y_{n,n}^*(T)&\mbox{in}&\Omega,
\end{array}\right.
\end{equation}
\begin{equation}
\left\{
\begin{array}{lll}
\p_t p_n^*+\triangle p_n^*=y_n^*-y_d&\mbox{in}&\Omega\times
(0,T),\\
p_n^*=0&\mbox{on}&\p \Omega\times (0,T),\\
p_n^*(0)=p_n^*(T)&\mbox{in}&\Omega ,
\end{array}\right.\label{Prin-2}
\end{equation}
and
\begin{equation}\label{Prin-3}
u_{i-1,n}^*=h_n\chi_\omega p_n^*(\tau_{i-1}),\;\;2\leq i\leq n,
\end{equation}
with
\begin{equation}\label{Prin-4}
y_n^*(0)= y_n^*(T),
\end{equation}
where $y_n^*\in L^\infty(0,T;L^2(\Omega))$ is defined by
\begin{equation}\label{def_yn*}
y_n^*(t)= y_{i,n}^*(t),\;\;t\in (\tau_{i-1},\tau_i],\;\;1\leq i\leq n,
\end{equation}
and $y_{i,n}^*\in C([\tau_{i-1},\tau_i]; L^2(\Omega))$, $1\leq i\leq n$.
\end{proposition}

Proposition \ref{Principle} is proved in Section \ref{sec_proofPrinciple}.

\begin{remark}
We could also consider the corresponding impulse version of the optimality system \eqref{Maxi-1}-\eqref{Maxi-2}-\eqref{Maxi-3}, but then its well-posedness would be hard to prove, and therefore, obtaining error estimates in such a way seems difficult.
\end{remark}

\paragraph{The approximating sampled-data optimal control problem $\SOCPn$.}
Let us now define the approximating sampled-data optimal control problem $\SOCPn$, for $n\geq 2$, by performing the usual sample-and-hold procedure on the control function.
This consists of freezing the value of the control over a certain horizon of time, usually called sampling time. In other words, we replace the control function $u\in L^2(0,T;L^2(\Omega))$ with a control that is piecewise constant in time, with values in $L^2(\Omega)$.
We keep the same notations as in the definition of $\IOCPn$, and we assume that the sampling time is equal to $h_n=T/n$. Recall that we have set $\tau_i=i\, h_n$, for $i=0,\ldots,n$.
We consider the class of sampled-data controls $f_n\in L^2(0,T;L^2(\Omega))$ defined by
\begin{equation}\label{def_fn}
f_n(t)=v_{i,n},\quad \forall\, t\in (\tau_{i-1},\tau_i],\quad 1\leq i\leq n,
\end{equation}
where $v_{i,n}\in L^2(\Omega)$ for every $i\in\{1,\ldots,n\}$. This class of controls is therefore identified with $(L^2(\Omega))^n$.

Recall that the functional $J$ is defined by \eqref{def_J}.
We consider the sampled-data optimal control problem $\SOCPn$, consisting of minimizing the functional
$$
J(y_n,f_n)=\frac{1}{2}\left(\int_0^T\|y_n-y_d\|^2\,\mathrm dt+
h_n\sum_{i=1}^n \|v_{i,n}\|^2\right)
$$
over all $(y_n,V_n)\in L^2(0,T;H^2(\Omega)\cap H_0^1(\Omega)) \cap H^1(0,T;L^2(\Omega)) \times (L^2(\Omega))^n$, with $V_n=(v_{1,n},\dots,v_{n,n})$, such that
\begin{equation*}
\left\{
\begin{array}{lll}
\partial_t y_n-\triangle y_{n}=\chi_\omega f_n&\mbox{in}&\Omega\times(0,T),\\
y_n=0&\mbox{on}&\partial\Omega\times (0,T),\\
y_n(0)=y_n(T)&\mbox{in}&\Omega,
\end{array}\right.
\end{equation*}
where $f_n\in L^2(0,T;L^2(\Omega))$ is the sampled control defined by \eqref{def_fn}.

\begin{proposition}\label{PMP_sampled}
For every $n\geq 2$, the optimal control problem $\SOCPn$ has a unique solution $(\bar{y}_n^*,V_n^*)$, with $V_n^*=(v_{1,n}^*,\dots,v_{n,n}^*)$.
The optimal solution $(\bar{y}_n^*,V_n^*)$ of $\SOCPn$ is characterized by the existence of $\bar{p}_n^*\in L^2(0,T;H^2(\Omega)\cap H_0^1(\Omega))\cap H^1(0,T;L^2(\Omega))$
such that
\begin{equation}\label{PMP_SD1}
\left\{
\begin{array}{lll}
\partial_t \bar{y}_n^*-\triangle \bar{y}_n^*=\chi_\omega f_n^*&\mbox{in}&\Omega\times (0,T),\\
\bar{y}_n^*=0&\mbox{on}&\partial\Omega\times (0,T),\\
\bar{y}_n^*(0)=\bar{y}_n^*(T)&\mbox{in}&\Omega,
\end{array}\right.
\end{equation}
\begin{equation}\label{PMP_SD2}
\left\{
\begin{array}{lll}
\partial_t \bar{p}_n^*+\triangle \bar{p}_n^*=\bar{y}_n^*-y_d&\mbox{in}&\Omega\times (0,T),\\
\bar{p}_n^*=0&\mbox{on}&\partial\Omega\times (0,T),\\
\bar{p}_n^*(0)=\bar{p}_n^*(T)&\mbox{in}&\Omega ,
\end{array}\right.
\end{equation}
and
\begin{equation}\label{PMP_SD3}
v_{i,n}^*=\frac{1}{h_n}\chi_\omega\int_{\tau_{i-1}}^{\tau_i} \bar{p}_n^*(t)\,\mathrm dt,\;\;\;1\leq i\leq n,
\end{equation}
where $f_n^*\in L^2(0,T;L^2(\Omega))$ is the (optimal) sampled-data control given by
\begin{equation}\label{def_fn*}
f_n^*(t)=v_{i,n}^*,~\forall\, t\in (\tau_{i-1},\tau_i],~1\leq i\leq n.
\end{equation}
\end{proposition}

Since the proof of Proposition \ref{PMP_sampled} is similar to the one of Proposition \ref{Principle}, we do not provide any proof in the present paper.
It is interesting to note that the optimal sampled-data control $f_n^*$, defined by \eqref{def_fn*}, is given by time-averages of the adjoint state $\bar{p}_n^*$ over the time-subdivision defined by the sampling time $h_n$ (see \eqref{PMP_SD3}). This fact has been proved as well in a more general context in \cite{BourdinTrelat}.

\begin{remark}\label{rem_sampled}
It is clear that the sampled-data optimal control problem $\SOCPn$ may be considered as an approximate version of $\OCP$, but it is less clear, at least intuitively, for the impulse optimal control problem $\IOCPn$. Before establishing precise error estimates in the next section, let us provide a first intuitive explanation.
Firstly, any continuously distributed control may be discretized by the sample-and-hold procedure, leading to the sampled-data control $\sum_{i=1}^n \chi_{((i-1)h_n,i\,h_n)} v_i$ with $v_i\in L^2(\Omega)$. Secondly, this sampled-data control can be seen as an approximation of an impulsive control, in the sense that
$$
\frac{1}{h_n} \chi_{(0,h_n)}\rightarrow\delta_{\{t=0\}}
$$
in the distributional sense, where $\delta_{\{t=0\}}$ is the Dirac mass at $t=0$ (note that we have as well the convergence of the corresponding solutions, see Lemma \ref{compare:3} further).
Denoting by $y(u)$ the solution of \eqref{eq_pb}, we have, noting that $u_{i-1,n}\simeq h_n u(i\,h_n)$, for $i=2,\dots,n$,
$$
y(u)\simeq y\left(\sum_{i=1}^n \chi_{((i-1)h_n,i\,h_n)}(\cdot) u(i\,h_n)\right)
  \simeq y\left(\sum_{i=2}^n \delta_{\{t= (i-1)\, h_n\}} h_n u(i\,h_n)\right)
  \simeq y\left(\sum_{i=2}^n \delta_{\{t= (i-1)\, h_n\}} u_{i-1,n}\right)
$$
and
$$
 \|u\|_{L^2(0,T;L^2(\Omega))}^2 \simeq \sum_{i=1}^n h_n \|u(i\,h_n)\|^2
 \simeq\frac{1}{h_n}\sum_{i=2}^n \|u_{i-1,n}\|^2.
$$
Here, $y\left(\sum_{i=2}^n \delta_{\{t= (i-1)\, h_n\}} u_{i-1,n}\right)$ is the corresponding impulsive solution, and we have moreover $J(y,u)\simeq J_n(Y_n,U_n)$.
\end{remark}

\subsection{Error estimates}
We keep all notations introduced in Section \ref{sec12}.
The main results of the paper are the following.

\paragraph{Error estimates for the impulse approximation.}

\begin{theorem}\label{Theorem-Error}
We set
\begin{equation}\label{error-2}
u_n^*(t)= \frac{1}{h_n} u_{i-1,n}^*,\quad t\in (\tau_{i-1},\tau_i],\;\;1\leq i\leq n,
\quad \;\;u_{0,n}^*= 0.
\end{equation}
Then there exists $C(T)>0$ such that
\begin{equation}\label{error_estim_control}
\|u^*-u_n^*\|_{L^2(0,T;L^2(\Omega))}\leq C(T) h_n^{1/2} \|y_d\|_{L^2(0,T;L^2(\Omega))}  ,
\end{equation}
and
\begin{equation}\label{wang-1}
|J_n(Y_n^*,U_n^*)-J(y^*,u^*)|\leq C(T) h_n^{1/2} \|y_d\|_{L^2(0,T;L^2(\Omega))}^2 .
\end{equation}
For every $p\in[2,+\infty)$, there exists $C(T,p)$ such that
\begin{equation}\label{Lp-estimate:1}
\|y^*-y_n^*\|_{L^p(0,T;L^2(\Omega))} \leq C(T,p) h_n^{1/p} \|y_d\|_{L^2(0,T;L^2(\Omega))} ,\quad p\in [2,+\infty).
\end{equation}
Moreover, the constant $C(T)$ and $C(T,p)$ are independent of $n$ and of $y_d$.
\end{theorem}

Theorem \ref{Theorem-Error} is proved in Section \ref{sec_proof_Theorem-Error}.
Note that we have assumed that $p<+\infty$ in the statement. If $p=+\infty$ then the situation is more complicated, and we have the following result.

\begin{theorem}\label{infty-estimate}
We assume that the subset $\omega$ of $\Omega$ has a $C^2$ boundary. Let $q\in (1,+\infty)$ be arbitrary.
If $\omega\not=\Omega$, then
\begin{equation}\label{infty-estimate:1}
\|y^*-y_n^*\|_{L^\infty(0,T;L^2(\Omega))}\leq\left\{
\begin{array}{lll}
C(T) h_n^{1/2N} \|y_d\|_{L^2(0,T;L^2(\Omega))} &\mbox{for}&N\geq 3,\\
C(T,q) h_n^{1/4q} \|y_d\|_{L^2(0,T;L^2(\Omega))} &\mbox{for}&N=2,\\
C(T) h_n^{1/4} \|y_d\|_{L^2(0,T;L^2(\Omega))} &\mbox{for}&N=1.
\end{array}\right.
\end{equation}
If $\omega=\Omega$, then
\begin{equation*}
\|y^*-y_n^*\|_{L^\infty(0,T;L^2(\Omega))}\leq C(T) h_n^{1/2} \|y_d\|_{L^2(0,T;L^2(\Omega))} .
\end{equation*}
The constant $C(T)$ and $C(T,q)$ are independent of $n$ and of $y_d$.
\end{theorem}

Theorem \ref{infty-estimate} is proved in Section \ref{sec_proof_infty-estimate}.

\begin{remark}\label{rem2}
The above error estimates are much easier to obtain when the control domain $\omega$ is equal to the whole domain $\Omega$, that is, when $\omega=\Omega$. But in this case the optimal control problems $\OCP$ and $\IOCPn$ have little interest.
Actually, the main difficulty in obtaining our results is due to the fact that, if $\omega\subsetneq\Omega$, then the function $\chi_\omega$ is not smooth and the function  $\chi_\omega p_n^*(\tau_{i-1})$ in \eqref{Prin-3}  is not in $H^1_0(\Omega)$. In the proofs of Theorems \ref{Theorem-Error} and \ref{infty-estimate} (see Section \ref{secproofs}), in addition to more or less standard functional analysis arguments, to overcome the abovementioned difficulty, we use smooth regularizations of the characteristic function $\chi_\omega$, the gradient of which we have to estimate in a refined way in some appropriate $L^p$ norm. Of course, this gradient blows up as the regularization parameter tends to zero, but fortunately there is some room to design appropriate regularizations, with adequate blow-up exponents (which we compute in a sharp way) that can be compensated elsewhere in the estimates, using Sobolev embeddings and usual functional inequalities. Using this approach, and deriving nonstandard estimates for the linear heat equation, we ultimately establish the desired error estimates.
\end{remark}

\begin{remark}
In Theorem \ref{infty-estimate}, if $\omega=\Omega$ (trivial case, according to Remark \ref{rem2}) then the order of convergence of the state is $1/2$, and we conjecture that it is sharp.\footnote{Actually, we are able to prove that the exponent $1/2$ is sharp in the estimate given in Lemma \ref{compare:3} (in Section \ref{sec_proof_Theorem-Error}) and in Lemma \ref{intuitive} (in Section \ref{sec_usefulestimate}), in the case where $\omega=\Omega$ and $p=2$. We do not provide the proofs of these facts here.}
If $\omega\subsetneq\Omega$, then we have obtained the error estimate \eqref{infty-estimate:1} but we conjecture that it is not sharp, and that the order of convergence $1/2$ should hold true as well.
\end{remark}

\paragraph{Error estimates for the sampled-data approximation.}

\begin{theorem}\label{thm_SD}
There exists $C(T)>0$ such that
\begin{equation}\label{errorSDcontrol}
\|u^*-f_n^*\|_{L^2(0,T;L^2(\Omega))} = \left(\sum_{i=1}^n \int_{\tau_{i-1}}^{\tau_i}\|u^*-v_{i,n}^*\|^2\,\mathrm dt\right)^{1/2}\leq C(T) h_n  \|y_d\|_{L^2(0,T;L^2(\Omega))} ,
\end{equation}
\begin{equation}\label{errorSDstate}
\|y^*-\bar{y}_n^*\|_{C([0,T];H_0^1(\Omega))}+
\|y^*-\bar{y}_n^*\|_{L^2(0,T;H^2(\Omega)\cap H_0^1(\Omega))\cap H^1(0,T;L^2(\Omega))}\leq C(T) h_n  \|y_d\|_{L^2(0,T;L^2(\Omega))} ,
\end{equation}
and
\begin{equation}\label{wang-2}
\vert J(y^*,u^*)-J(\bar{y}_n^*,f_n^*)\vert\leq C(T)h_n \|y_d\|_{L^2(0,T;L^2(\Omega))}^2 .
\end{equation}
Moreover, the constant $C(T)$ is independent of $n$ and of $y_d$.
\end{theorem}

Theorem \ref{thm_SD} is proved in Section \ref{sec_proof_thm_SD}.

\begin{remark}
It is interesting to note that the error estimates are better for the sampled-data approximation than for the impulse approximation. For instance, the control error estimate is of order $1$ for the sampled-data approximation, but is of order $1/2$ for the impulse approximation (in $L^2$ norm). This is not surprising because, as explained in Remark \ref{rem_sampled}, the sampled-data optimal control problem $\SOCPn$ can easily be recast as a classical approximation of $\OCP$, in the sense that the class of admissible sampled controls is a subset of the class of admissible controls of $\OCP$. In this sense, obtaining the error estimates of Theorem \ref{thm_SD} could be expected. In contrast, the set of unknowns $(Y_n,U_n)$ for the impulse optimal control problem $\IOCPn$ is not a subset of the set of unknowns $(y,u)$ for $\OCP$. This explains why the derivation of error estimates for $\IOCPn$ is much more difficult.
\end{remark}

\subsection{Further comments}
We have established error estimates for the optimal controls and states of impulse approximations and of sampled-data approximations of a linear quadratic optimal control problem for a linear heat equation, with internal control, and with periodic terminal constraints.
To our knowledge, this is the first result providing such convergence results and estimates, in an infinite-dimensional context.
Many questions are open, that are in order.

\paragraph{Terminal constraints.}
Here, we have considered periodic terminal constraints. This condition is instrumental in order to obtain existence and uniqueness results and to be able to derive a PMP (see, in particular, Lemma \ref{Energy} in Section \ref{sec21}). But it is of course of interest to consider other terminal conditions.
For instance, one may want to consider the problem $\OCP$ with the fixed terminal conditions $y(0)=y^0\in L^2(\Omega)$ and $y(T)=0$. It is well known that this exact null controllability problem admits some solutions, without any specific requirement on the (open) domain of control $\omega$. But it is well known too that  the final adjoint state coming from the PMP lives in a very big space that is larger than any distribution space. This raises an important difficulty from the functional analysis point of view, preventing us from extending our analysis to this setting.

Moreover, when considering more general equations (see the next item), if one considers an infinite-codimensional state constraint then it is well known that the PMP may fail (see \cite{LiYong}), and then in this case even the basic fact of establishing an optimality system may raise some impassable obstacles.

\paragraph{More general evolution equations.}
We have considered the linear homogeneous heat equations. The question is open to extend our analysis to more general parabolic equations, of the kind $\partial_ty=Ay+\chi_\omega u$, with $A:D(A)\rightarrow L^2(\Omega)$ generating an analytic semi-group. For instance, one may want to replace the Dirichlet Laplacian with a general elliptic second-order differential operator, with various possible boundary conditions (Dirichlet, Neumann, Robin, mixed), or with the Stokes operator. It is likely that our results may be extended to this situation, but note anyway that, in our proofs, we repeatedly use the fact that we have Dirichlet conditions.

More generally, the question is open to investigate semilinear parabolic equations, of the kind $\partial_ty=Ay+f(y)+\chi_\omega u$. Even when $A$ is the Dirichlet Laplacian, this extension seems to be challenging.

The case of hyperbolic equations is another completely open issue. Certainly, the first case to be investigated is the wave equation: in that case one replaces the first equation in \eqref{eq_pb} with the internally controlled linear homogeneous wave equation $\partial_{tt}y=\triangle y+\chi_\omega u$. In this case, it is well known that exact controllability holds true under the so-called Geometric Control Condition on $(\omega,T)$. What happens for the corresponding impulse model is open, and is far from being clear (see \cite{Khapalov} for results in that direction). Also, many estimates, that are quite standard for heat-like equations and that we use in this paper, are not valid anymore in the hyperbolic context.

\paragraph{More general control operators.}
In this paper, we have consider an internal control. Writing the control system in the abstract form $\partial_ty=Ay+Bu$, this corresponds to considering the control operator $B$ defined by $Bu=\chi_\omega u$. In this case, the control operator is bounded, and we implicitly use this fact in many places in our proofs. We expect that our results can be extended to more general classes of bounded control operators, but the case of unbounded control operators seems much more challenging. For instance, what happens when considering a Dirichlet boundary control is open.

\paragraph{Time-varying control domains and optimal design.}
Another open question is to derive our error estimates for time-varying control domains. Note that control issues for wave equations with time-varying domains have been investigated in \cite{Lebeau}. In our context, this means that we consider a control domain $\omega(t)$ depending on $t$ in \eqref{eq_pb}. In this case, the definition of the approximating impulse control system \eqref{impulse} must be adapted as well, by considering $\omega(\tau_{i-1})$ at time $\tau_{i-1}$. It is likely that our main results may be at least extended to the case where $\omega(t)$ depends continuously on $t$. The general case is open.

Related to this issue, is the question of determining how to place and shape ``optimally" the control domain. Of course, the optimization criterion has to be defined, and we refer to \cite{PTZObs1, PTZ_HUM, PTZ_optparab} where optimal design problems have been modeled and studied. In the context of the present paper, we could investigate the problem of designing the best possible control domain such that the constants appearing in our error estimates be minimal.

Let us be more precise and let us define the open problem.
Given any $T>0$, and any open subset $\omega$ of $\Omega$, Theorem \ref{Theorem-Error} asserts that there exists a constant $C(T)>0$ such that the error estimates \eqref{error-2}, \eqref{error_estim_control} and \eqref{Lp-estimate:1} (with $p=2$, for instance) are satisfied. Since this constant depends on $\omega$, we rather denote it by $C_T(\omega)$.
Given a real number $L\in(0,1)$, we consider the optimal design problem
$$
\inf_{\omega\in\mathcal{U}_L} C_T(\omega),
$$
that is, the problem of finding, if it exists, the best possible control subset having a prescribed Lebesgue measure, such that the functional constant in the error estimates is as small as possible. This prescribed measure is $L\vert\Omega\vert$, that is, a fixed fraction of the total volume of the domain.
We stress that the set of unknowns is the (very big) set of all possible measurable subsets of $\Omega$ of measure $L\vert\Omega\vert$. It does not share any good compactness properties that would be appropriate for deriving nice functional properties, and thus already the problem of the existence of an optimal set is far from obvious. However, following \cite{PTZ_optparab} where similar optimal design problems have been investigated in the parabolic setting, we conjecture that there exists a unique best control domain, in the sense given above. Proving this conjecture, and deriving characterizations of the optimal set, is an interesting open issue.

Note that, less ambitiously than the problem above, one could already consider simpler optimal design problems, where the problem consists, for instance, of optimizing the placement of a control domain having a prescribed shape, such as a ball: in this case the set of unknowns is finite-dimensional (centers of the balls).

\paragraph{Impulse Riccati theory.}
In the present paper, we have considered a problem within a finite horizon of time $T$. It would be interesting to consider the optimal control problems $\OCP$, $\IOCPn$ and $\SOCPn$ in infinite horizon, that is, when $T=+\infty$. In this case, the optimal control solution of $\OCP$ is obtained by the well known Riccati theory (see, e.g., \cite{Zabczyk}), which gives, here,
$$
u = \chi_\omega E(y-y_d) ,
$$
where $E\in L(L^2(\Omega))$ is the unique negative definite solution of the algebraic Riccati equation
$$
\triangle E+E\triangle+E\chi_\omega E=\mathrm{id} .
$$

For the approximating impulse problem $\IOCPn$, up to our knowledge, the Riccati procedure has not been investigated. In other words, up to now there does not seem to exist a Riccati theory for impulse linear quadratic optimal control problems in infinite dimension. Developing such a theory is already a challenge in itself. Assuming that such a theory has been established, the next challenge would be to establish as well the corresponding error estimates on the control and on the state, as done in our paper.

For the approximating sampled-data problem $\SOCPn$, few results exist in the literature. In \cite{RosenWang} the authors have established a convergence result (which can certainly be improved, by combining it with the more recent results of \cite{LasieckaTriggiani1, LiuZheng} for instance), but we are not aware of any result providing error estimates as in our paper.

For impulse systems in particular, such a theory would certainly be very useful for many practical issues, because, as already mentioned, impulse control may be an interesting alternative to discretization approaches, or to sample-and-hold procedures, which is sometimes better suited to the context of the study.
Notice that, although the theory of space semi-discretization of the Riccati procedure is complete in the parabolic case (but not in the hyperbolic case when the control operators are unbounded), to our knowledge the theory is far from complete for infinite-dimensional sampled-data control systems. Therefore, with respect to sample-and-hold procedures, this is one more motivation for developing an impulse Riccati theory and its approximations.

\section{Proofs}\label{secproofs}

\subsection{Preliminaries}\label{sec21}
\paragraph{Existence and uniqueness.}
We start with an easy existence and uniqueness result, together with a useful estimate.
Throughout the paper, we denote by  $\{e^{t\triangle}\}_{t\geq 0}$  the semi-group generated by the Dirichlet-Laplacian on $L^2(\Omega)$.

\begin{lemma}\label{Energy}
Let $T>0$ be arbitrary.
Let $f\in L^2(0,T;L^2(\Omega))$. Then the equation
\begin{equation}\label{Energy:1}
\left\{
\begin{array}{lll}
\p_t y-\triangle y=f&\mbox{in}&\Omega\times (0,T),\\
y=0&\mbox{on}&\p\Omega\times (0,T),\\
y(0)=y(T)&\mbox{in}&\Omega
\end{array}
\right.
\end{equation}
has a unique solution $y\in H^1(0,T;L^2(\Omega))\cap L^2(0,T;H^2(\Omega)\cap H_0^1(\Omega))$.
Moreover, there exists $C(T)>0$, not depending on $f$ and on $y$, such that
\begin{equation}\label{Energy:2}
\|y\|_{C([0,T];H_0^1(\Omega))}+\|y\|_{H^1(0,T;L^2(\Omega))\cap L^2(0,T;H^2(\Omega)\cap H_0^1(\Omega))}
\leq C(T)\|f\|_{L^2(0,T;L^2(\Omega))}.
\end{equation}
\end{lemma}

\begin{proof}
As a preliminary remark, we recall that, given $y_0\in L^2(\Omega)$ and $f\in L^2(0,T;L^2(\Omega))$ arbitrary, there exists a unique weak solution $y(\cdot;y_0,f)\in L^2(0,T;H_0^1(\Omega))\cap H^1(0,T;H^{-1}(\Omega))\subset C([0,T];L^2(\Omega))$ of $\partial_t y - \triangle y=f $ in $\Omega\times (0,T)$, with $y=0$ along $\partial\Omega\times (0,T)$, such that $y(0)=y_0$ (see \cite{LM} for instance). Here, ``weak" means that the differential equation is written in $H^{-1}(\Omega)$. Moreover, if $y_0\in H_0^1(\Omega)$, then $y(\cdot;y_0,f)\in L^2(0,T;H^2(\Omega)\cap H_0^1(\Omega))
\cap H^1(0,T;L^2(\Omega))\subset C([0,T];H_0^1(\Omega))$.

Given any $f\in L^2(0,T;L^2(\Omega))$, let us prove the existence and uniqueness of a weak solution of \eqref{Energy:1}.
Since $\|e^{T \triangle}\|_{\mathcal L(L^2(\Omega),L^2(\Omega))}\leq e^{-\lambda_1 T}<1$,
it follows that $(I-e^{T\triangle})^{-1}$  exists and $\|(I-e^{T\triangle})^{-1}\|_{\mathcal L(L^2(\Omega),L^2(\Omega))}\leq (1-e^{-\lambda_1 T})^{-1}$,
where $-\lambda_1<0$ is the first eigenvalue of the Dirichlet Laplacian.
Now  we define
\begin{eqnarray}\label{p-2}
 y_0^f= (I-e^{T\triangle})^{-1} \int_0^T e^{ (T-t)\triangle} f(t) \,\mathrm dt ,
\end{eqnarray}
and
\begin{eqnarray}\label{p-3}
 y(t;y_0^f,f)=e^{t\triangle} y_0^f + \int_0^t e^{(t-s)\triangle} f(s)\, \mathrm ds,\quad t\in[0,T].
\end{eqnarray}
Then $y_0^f\in L^2(\Omega)$ and
$y(\cdot;y_0^f,f)\in L^2(0,T;H_0^1(\Omega))\cap H^1(0,T;H^{-1}(\Omega))$ is the weak solution of $\partial_t y - \triangle y=f$ in $\Omega\times (0,T)$, with $y=0$ along $\partial\Omega\times (0,T)$, such that $y(0)=y_0^f$.
Using \eqref{p-2} and \eqref{p-3}, we have
\begin{multline*}
 y(T;y_0^f,f) =e^{T\triangle} y_0^f + \int_0^T e^{(T-t)\triangle} f(t)\, \mathrm dt
 = e^{T\triangle} (I-e^{T\triangle})^{-1} \int_0^T e^{(T-t)\triangle} f(t) \,\mathrm dt  + \int_0^T e^{ (T-t)\triangle} f(t)\, \mathrm dt \\
 =(I-e^{T\triangle})^{-1} \int_0^T e^{(T-t)\triangle} f(t) \,\mathrm dt =y(0;y_0^f,f),
\end{multline*}
which gives the periodicity requirement. Hence $y(\cdot;y_0^f,f)$ is a weak solution of \eqref{Energy:1}.

Now, if $y_1$ and $y_2$ are two weak solutions of \eqref{Energy:1} associated with $f$, then
\begin{eqnarray*}
\frac{1}{2} \frac{\mathrm d}{\mathrm dt} \|y_1(t)-y_2(t)\|^2_{L^2(\Omega)}
 + \int_\Omega |\nabla y_1(t)-\nabla y_2(t) |^2 \mathrm dx =0,\;\;\mbox{a.e.}\;\; t\in(0,T).
\end{eqnarray*}
Integrating the latter equality over $(0,T)$, we deduce from the periodicity condition that $y_1=y_2$.
Therefore the weak solution is unique.

It remains to prove that the weak solution $y$ of \eqref{Energy:1} actually belongs to
$L^2(0,T;H^2(\Omega)\cap H_0^1(\Omega))\cap H^1(0,T;L^2(\Omega))$ and to prove the estimate \eqref{Energy:2}.
Using the preliminary remark, we have $y(T)\in H_0^1(\Omega)$, and since $y(0)=y(T)$, it follows that $y(0)\in H_0^1(\Omega)$. Therefore
$y\in L^2(0,T;H^2(\Omega)\cap H_0^1(\Omega))\cap H^1(0,T;L^2(\Omega))$.
Now, multiplying the differential equation by $2y$ and integrating over $\Omega$, we get that
\begin{eqnarray}\label{add-1}
\frac{\mathrm d}{\mathrm dt} \|y\|^2
 + 2\|\nabla y\|^2=2\langle f,y\rangle,\;\;\mbox{a.e.}\;\; t\in (0,T).
\end{eqnarray}
Using the Poincar\'{e} inequality  $\|\varphi\|\leq C\|\nabla\varphi\|$,  valid for any $\varphi\in H_0^1(\Omega)$, combined with \eqref{add-1} and the Young inequality, we infer that
   \begin{eqnarray*}
          \frac{\mathrm d}{\mathrm dt} \|y\|^2
 + 2\|\nabla y\|^2\leq \|\nabla y\|^2+C\|f\|^2, \;\;\mbox{a.e.}\;\; t\in (0,T).
        \end{eqnarray*}
Here and throughout, $C$ designates a generic positive constant only depending on $\Omega$.
Integrating over $(0,T)$, we obtain that
\begin{equation}\label{Energy:3}
\int_0^T \|\nabla y\|^2\,\mathrm dt\leq C\int_0^T \|f\|^2\,\mathrm dt.
\end{equation}
Besides, multiplying the first equation of \eqref{Energy:1} by $-2t\triangle y$ and integrating
over $\Omega$, we have
\begin{equation*}
t\f{\mathrm d}{\mathrm dt}\|\nabla y\|^2+2t\|\triangle y\|^2=-2t\langle f,\triangle y\rangle\leq t\|\triangle y\|^2+t\|f\|^2.
\end{equation*}
Integrating again over $(0,T)$, we obtain that
  \begin{equation*}
       T\|\nabla y(T)\|^2\leq \int_0^T \|\nabla y\|^2\,\mathrm dt+  T \int_0^T \|f\|^2\,\mathrm dt,
       \end{equation*}
which, combined with (\ref{Energy:3}) and the third equation of (\ref{Energy:1}), gives
$\|\nabla y(0)\|^2\leq  \frac{C+T}{T} \int_0^T \|f\|^2\,\mathrm dt$.
This, together with the first and the second equations of (\ref{Energy:1}), implies that
\begin{multline*}
\|y\|_{C([0,T];H_0^1(\Omega))}+\|y\|_{H^1(0,T;L^2(\Omega))\cap L^2(0,T;H^2(\Omega)\cap H_0^1(\Omega))}\\
   \leq C(\|\nabla y(0)\|+\|f\|_{L^2(0,T;L^2(\Omega))})\leq  \frac{C(T+1)}{T} \|f\|_{L^2(0,T;L^2(\Omega))}.
\end{multline*}
This completes the proof.
\end{proof}

\paragraph{Optimality system (PMP).}
The proof of existence and uniqueness of an optimal solution of $\OCP$ is easy. Since it is similar, but simpler, than the proof of Proposition \ref{Principle} further, we skip it.

Let $(y^*,u^*)$ be the optimal solution of $\OCP$. For any $v\in L^2(0,T;L^2(\Omega))$ and $\lambda\in\mathbb{R}\setminus\{0\}$, let $y_{\lambda,v}$ be the solution of
\begin{eqnarray*}
\left\{\begin{array}{lll}
\partial_t y_{\lambda,v}-\triangle y_{\lambda,v}=\chi_\omega (u^*+\lambda v)&\mbox{in}&\Omega\times
(0,T),\\
y_{\lambda,v}=0&\mbox{on}&\partial\Omega\times (0,T),\\
y_{\lambda,v}(0)=y_{\lambda,v}(T)&\mbox{in}&\Omega.
\end{array}\right.
\end{eqnarray*}
Setting $z=\dfrac{y_{\lambda,v}-y^*}{\lambda}$, we have
\begin{eqnarray*}
\left\{\begin{array}{lll}
\partial_t z-\triangle z=\chi_\omega v&\mbox{in}&\Omega\times
(0,T),\\
z=0&\mbox{on}&\partial\Omega\times (0,T),\\
z(0)=z(T)&\mbox{in}&\Omega.
\end{array}\right.
\end{eqnarray*}
Moreover, by definition, $J(y_{\lambda,v},u^*+\lambda v)\geq J(y^*,u^*)$, for every $\lambda\neq 0$, and hence
\begin{eqnarray}\label{e-l-11}
 \int_0^T \int_\Omega (y^*-y_d) z \,\mathrm dx\,\mathrm dt+\int_0^T \int_\Omega u^*v \,\mathrm dx\,\mathrm dt=0.
\end{eqnarray}
Let $p$ be the solution of
\begin{eqnarray*}
\left\{\begin{array}{lll}
\partial_t p+\triangle p=y^*-y_d&\mbox{in}&\Omega\times
(0,T),\\
p=0&\mbox{on}&\partial\Omega\times (0,T),\\
p(0)=p(T)&\mbox{in}&\Omega.
\end{array}\right.
\end{eqnarray*}
This, together with \eqref{e-l-11}, yields that
\begin{eqnarray*}
0=\int_0^T \int_\Omega (y^*-y_d) z \mathrm dx\mathrm dt
+\int_0^T \int_\Omega u^*v \mathrm dx\mathrm dt=\int_0^T \int_\Omega (-\chi_\omega p+u^*)v\mathrm dx\mathrm dt.
\end{eqnarray*}
Hence $ u^*=\chi_\omega p$.
This gives the PMP for the problem $\OCP$.

\subsection{Proof of Proposition \ref{Principle}}\label{sec_proofPrinciple}
Let us first prove the existence and uniqueness of a solution of $\IOCPn$.
According to the beginning of the proof of Lemma \ref{Energy}, we first recall that for $U_n=(u_{1,n},\dots, u_{n-1,n})\in (L^2(\Omega))^{n-1}$, we say that $Y_n=(y_{1,n},\dots, y_{n,n})\in X$ is a weak solution of \eqref{impulse} if
\begin{eqnarray*}
 \langle \partial_t y_{i,n}(t),\varphi \rangle_{H^{-1}(\Omega),H_0^1(\Omega)}
 + \int_\Omega \nabla y_{i,n}(t) \cdot \nabla\varphi \,\mathrm dx
 =0\,\mbox{ a.e. } t\in (\tau_{i-1},\tau_i),\;1\leq i\leq n,
\end{eqnarray*}
for each $\varphi\in H_0^1(\Omega)$ (this means that the differential equation is written in $H^{-1}(\Omega)$) and $y_{i,n}(\tau_{i-1})=y_{i-1,n}(\tau_{i-1})+\chi_\omega u_{i-1,n}$, for $2\leq i\leq n$, and $y_{1,n}(0)=y_{1,n}(T)$.
It is then a standard fact that \eqref{impulse} has a unique weak solution.

Let $d^*=\inf J_n(Y_n,U_n)\geq 0$, where the infimum is taken over all pairs $(Y_n,U_n)\in X\times (L^2(\Omega))^{n-1}$ satisfying \eqref{impulse}. By definition, there exists a sequence $(Y_{n,m},U_{n,m})_{m\geq 1}$, with
$Y_{n,m}=(y_{i,n,m})_{1\leq i\leq n}$ and $U_{n,m}=(u_{i,n,m})_{1\leq i\leq n-1}$ satisfying \eqref{impulse}, such that
\begin{equation}\label{Unique-1}
d^*\leq J_n(Y_{n,m},U_{n,m}) \leq d^*+\f{1}{m}.
\end{equation}
Integrating the equations given by \eqref{impulse}, we get
\begin{equation}\label{Unique-3}
\left\{
\begin{array}{l}
y_{1,n,m}(t)=e^{t\triangle} y_{1,n,m}(0),\quad t\in [0,\tau_1],\\
y_{i,n,m}(t)=e^{t\triangle} y_{1,n,m}(0)+\d{\sum_{j=2}^i} e^{(t-\tau_{j-1})\triangle}\chi_\omega u_{j-1,n,m},
\quad t\in [\tau_{i-1},\tau_i],\quad 2\leq i\leq n,\\
y_{1,n,m}(0)=e^{T\triangle} y_{1,n,m}(0)+\d{\sum_{j=2}^n} e^{(T-\tau_{j-1})\triangle}\chi_\omega u_{j-1,n,m}.
\end{array}\right.
\end{equation}
Using \eqref{Unique-1} and the third equality of \eqref{Unique-3}, we infer that
\begin{equation}\label{Unique-4}
\|y_{1,n,m}(0)\|\leq C,
\end{equation}
for every $m\geq 1$.%
\footnote{Here, the $L^2$ norm is used. For $y_{1,n,m}(0)$, we may wish to consider the $H_0^1(\Omega)$ norm. But since $y_{i,n,m}(\tau_{i-1})=y_{i-1,n,m}(\tau_{i-1})+\chi_\omega u_{i-1,n}$ and $u_{i-1,n}\in L^2(\Omega)$ ($2\leq i\leq n$), it follows that $y_{i,n,m}(\tau_{i-1})\in L^2(\Omega)$, for $2\leq i\leq n$. Hence the $H_0^1(\Omega)$ norm does not seem to be useful.}
Here and throughout the proof, $C$ designates a generic positive constant not depending of $m$. Multiplying the first equation of \eqref{impulse} (written for $y_{i,n,m}$) by $2y_{i,n,m}$ and integrating over $\Omega\times (\tau_{i-1},t)$, we obtain that
\begin{equation*}
\|y_{i,n,m}(t)\|^2+2\int_{\tau_{i-1}}^t \|\nabla y_{i,n,m}(s)\|^2\,\mathrm ds\leq \|y_{i,n,m}(\tau_{i-1})\|^2,
\quad \forall t\in [\tau_{i-1},\tau_i],\quad 1\leq i\leq n,
\end{equation*}
which implies that
\begin{equation*}
\|y_{i,n,m}\|_{C([\tau_{i-1},\tau_i];L^2(\Omega))}+\|y_{i,n,m}\|_{X_i}\leq C\|y_{i,n,m}(\tau_{i-1})\|,\quad 1\leq i\leq n.
\end{equation*}
This, together with (\ref{Unique-1}), (\ref{Unique-4}) and the third equation of (\ref{impulse}), gives
\begin{equation*}
\sum_{i=1}^n (\|y_{i,n,m}\|_{C([\tau_{i-1},\tau_i];L^2(\Omega))}+\|y_{i,n,m}\|_{X_i})
+\sum_{i=2}^n \|u_{i-1,n,m}\|\leq C.
\end{equation*}
Hence, up to some subsequence, we have
\begin{equation*}
y_{i,n,m}\ra y_{i,n}^*\;\;\mbox{weakly in}\;\;X_i,\;\;\mbox{strongly in}\;\;L^2(\tau_{i-1},\tau_i;L^2(\Omega)),\;\;1\leq i\leq n,
\end{equation*}
and
\begin{equation*}
u_{i-1,n,m}\ra u_{i-1,n}^*\;\;\mbox{weakly in}\;\;L^2(\Omega),\;\;2\leq i\leq n,
\end{equation*}
for some $Y_n^*=(y_{1,n}^*,y_{2,n}^*,\ldots,y_{n,n}^*)\in X$ and $U_n^*=(u_{1,n}^*,u_{2,n}^*,\ldots,u_{n-1,n}^*)\in (L^2(\Omega))^{n-1}$.
Passing to the limit in \eqref{Unique-1} and in \eqref{impulse}, it is clear that $(Y_n^*,U_n^*)$ is an optimal solution of $\IOCPn$.

The uniqueness follows from the strict convexity of the functional $\widetilde{J}_n: (L^2(\Omega))^{n-1}\ra [0,+\infty)$ defined by $\widetilde{J}_n(U_n)=J_n(Y_n,U_n)$,
where $Y_n$ is the unique solution of (\ref{impulse}) corresponding to $U_n$.

\medskip

Let us now prove the characterization of the optimal solution given in the proposition.

We assume that $(Y_n^*,U_n^*)$ is the optimal solution of $\IOCPn$. Let us prove the existence of the adjoint state. The argument goes by perturbation of the optimal solution. Given any $U_n=(u_{1,n},u_{2,n},\ldots,u_{n-1,n})\in (L^2(\Omega))^{n-1}$ and any $\l\in (0,1)$, we set
\begin{equation}\label{Prin-5}
U_{n,\l}= U_n^*+\l U_n .
\end{equation}
Let $Y_{n,\l}=(y_{1,n,\l},y_{2,n,\l},\ldots,y_{n,n,\l})$ be the solution of
\begin{equation*}
\left\{\begin{array}{lll}
\p_t y_{i,n,\l}-\triangle y_{i,n,\l}=0&\mbox{in}&\Omega\times
(\tau_{i-1},\tau_i),\quad 1\leq i\leq n,\\
y_{i,n,\l}=0&\mbox{on}&\p\Omega\times (\tau_{i-1},\tau_i),\quad 1\leq i\leq n,\\
y_{i,n,\l}(\tau_{i-1})=y_{i-1,n,\l}(\tau_{i-1})+\chi_\omega (u_{i-1,n}^*+\l u_{i-1,n})&\mbox{in}&\Omega,\quad 2\leq i\leq n,\\
y_{1,n,\l}(0)=y_{n,n,\l}(T)&\mbox{in}&\Omega.
\end{array}\right.
\end{equation*}
Setting $z_{i,n}= \f{y_{i,n,\l}-y_{i,n}^*}{\l}$, $1\leq i\leq n$, we have
\begin{equation}\label{Prin-7}
\left\{\begin{array}{lll}
\p_t z_{i,n}-\triangle z_{i,n}=0&\mbox{in}&\Omega\times
(\tau_{i-1},\tau_i),\quad 1\leq i\leq n,\\
z_{i,n}=0&\mbox{on}&\p\Omega\times (\tau_{i-1},\tau_i),\quad 1\leq i\leq n,\\
z_{i,n}(\tau_{i-1})=z_{i-1,n}(\tau_{i-1})+\chi_\omega u_{i-1,n}&\mbox{in}&\Omega,\quad 2\leq i\leq n,\\
z_{1,n}(0)=z_{n,n}(T)&\mbox{in}&\Omega.
\end{array}\right.
\end{equation}
Since $(Y_n^*,U_n^*)$ is the optimal solution of $\IOCPn$, we have $J_n(Y_{n,\l},U_{n,\l})-J_n(Y_n^*,U_n^*)\geq 0$. Dividing by $\lambda$ and passing the limit $\l\ra 0^+$, using \eqref{Penalty}, \eqref{Prin-5} and \eqref{Prin-7}, we infer that
\begin{equation}\label{Prin-8}
\d{\sum_{i=1}^n \int_{\tau_{i-1}}^{\tau_i}} \langle y_{i,n}^*-y_d,z_{i,n}\rangle \,\mathrm dt
+\d{\f{1}{h_n}\sum_{i=2}^n} \langle u_{i-1,n}^*,u_{i-1,n}\rangle \geq 0,\qquad \forall\; U_n\in (L^2(\Omega))^{n-1}.
\end{equation}
Let $p_n^*$ be defined by \eqref{Prin-2} (same reasoning as in Section \ref{sec21}).
Multiplying the first equation of
(\ref{Prin-7}) by $p_n^*$ and integrating over $\Omega\times (\tau_{i-1},\tau_i)$, we get\begin{equation*}
\langle z_{i,n}(\tau_i),p_n^*(\tau_i)\rangle -\langle z_{i,n}(\tau_{i-1}),p_n^*(\tau_{i-1})\rangle
=\int_{\tau_{i-1}}^{\tau_i} \langle y_{i,n}^*-y_d,z_{i,n}\rangle \,\mathrm dt,\quad 1\leq i\leq n ,
\end{equation*}
and summing over $i=1,2,\ldots,n$, we obtain
\begin{equation}\label{Prin-9}
\sum_{i=1}^n \int_{\tau_{i-1}}^{\tau_i} \langle y_{i,n}^*-y_d,z_{i,n}\rangle \,\mathrm dt
=-\sum_{i=2}^n \langle \chi_\omega p_n^*(\tau_{i-1}),u_{i-1,n}\rangle ,
\end{equation}
which, combined with (\ref{Prin-8}), yields (\ref{Prin-3}).

\medskip

Let us now prove the converse, that is, let us prove that, if $(Y_n^*,U_n^*)$ and $p_n^*$ satisfy \eqref{Prin-1}-\eqref{Prin-2}-\eqref{Prin-3}-\eqref{Prin-4}, then $(Y_n^*,U_n^*)$ is the optimal solution of $\IOCPn$.

Given any $U_n=(u_{1,n},u_{2,n},\ldots,u_{n-1,n})\in (L^2(\Omega))^{n-1}$, we denote by $
Y_n=(y_{1,n},y_{2,n},\ldots,y_{n,n})$ the corresponding solution of \eqref{impulse}. By using arguments similar to those used to establish \eqref{Prin-9}, we obtain that
\begin{equation*}
\sum_{i=1}^n \int_{\tau_{i-1}}^{\tau_i} \langle y_{i,n}^*-y_d,y_{i,n}-y_{i,n}^*\rangle \,\mathrm dt
+\sum_{i=2}^n \langle \chi_\omega p_n^*(\tau_{i-1}),u_{i-1,n}-u_{i-1,n}^*\rangle =0.
\end{equation*}
This, together with (\ref{Prin-3}), implies that
\begin{equation*}
\sum_{i=1}^n \int_{\tau_{i-1}}^{\tau_i} \langle y_{i,n}^*-y_d,y_{i,n}-y_{i,n}^*\rangle \,\mathrm dt
+\f{1}{h_n}\sum_{i=2}^n \langle u_{i-1,n}^*,u_{i-1,n}-u_{i-1,n}^*\rangle =0.
\end{equation*}
Hence
\begin{multline*}
J_n(Y_n,U_n)-J_n(Y_n^*,U_n^*)\\
=\d{\f{1}{2}\sum_{i=1}^n\int_{\tau_{i-1}}^{\tau_i}} \langle y_{i,n}+y_{i,n}^*-2y_d,y_{i,n}-y_{i,n}^*\rangle \,\mathrm dt
+\d{\f{1}{2h_n}\sum_{i=2}^n} \langle u_{i-1,n}+u_{i-1,n}^*,u_{i-1,n}-u_{i-1,n}^*\rangle \\
\geq\d{\sum_{i=1}^n \int_{\tau_{i-1}}^{\tau_i}} \langle y_{i,n}^*-y_d,y_{i,n}-y_{i,n}^*\rangle \,\mathrm dt
+\d{\f{1}{h_n}\sum_{i=2}^n} \langle u_{i-1,n}^*,u_{i-1,n}-u_{i-1,n}^*\rangle
=0.
\end{multline*}
We conclude that $(Y_n^*,U_n^*)$ is the optimal solution of $\IOCPn$.

\subsection{Proof of Theorem \ref{Theorem-Error}}\label{sec_proof_Theorem-Error}

\subsubsection{A first estimate}
The following lemma compares two states generated by controls activated in different ways.

\begin{lemma}\label{compare:3}
Let $0\leq T_1<T_1+\delta<T_2<+\infty$ and let $u\in L^2(\Omega)$. Let $z$ and $w$ be the solutions of
\begin{equation}\label{compare:1}
\left\{\begin{array}{lll}
\p_t z-\triangle z=\d{\f{1}{\delta}}\chi_{(T_1,T_1+\delta)}\chi_\omega u&\mbox{in}&\Omega\times
(T_1,T_2),\\
z=0&\mbox{on}&\p \Omega\times (T_1,T_2),\\
z(T_1)=0 &\mbox{in}&\Omega
\end{array}\right.
\end{equation}
and
\begin{equation}\label{compare:2}
\left\{\begin{array}{lll}
\p_t w-\triangle w=0&\mbox{in}&\Omega\times
(T_1,T_2),\\
w=0&\mbox{on}&\p \Omega\times (T_1,T_2),\\
w(T_1)=\chi_\omega u &\mbox{in}&\Omega.
\end{array}\right.
\end{equation}
For every $p\in[2,\infty)$, there exists $C(T_2,p)>0$ such that
\begin{equation*}
\|z-w\|_{L^p(T_1,T_2;L^2(\Omega))}\leq C(T_2,p) \delta^{\f{1}{p}} \|\chi_\omega u\|.
\end{equation*}
\end{lemma}

\begin{proof}
By the  definitions  of $z$ and $w$, we have
\begin{equation}\label{compare:5}
 z(t)-w(t)=\int_{T_1}^{t} e^{\triangle(t-\tau)}\f{1}{\delta}\chi_{(T_1,T_1+\delta)}(\tau)\chi_\omega u\,\mathrm d\tau
 -e^{\triangle(t-T_1)} \chi_\omega u,\quad \forall\;t\in [T_1,T_2].
\end{equation}
Let $q=\f{p}{p-1}$ and let $f\in L^q(T_1,T_2;L^2(\Omega))$. Let $\varphi$ be the solution of
\begin{equation}\label{compare:6}
\left\{\begin{array}{lll}
\p_t \varphi+\triangle \varphi=f&\mbox{in}&\Omega\times
(T_1,T_2),\\
\varphi=0&\mbox{on}&\p \Omega\times (T_1,T_2),\\
\varphi(T_2)=0 &\mbox{in}&\Omega.
\end{array}\right.
\end{equation}
By \cite[Theorem 1]{Lamberton}, there exists $C(T_2,p)>0$ such that
\begin{equation}\label{compare:7}
\|\p_t \varphi\|_{L^q(T_1,T_2;L^2(\Omega))} \leq C(T_2,p)\|f\|_{L^q(T_1,T_2;L^2(\Omega))}.
\end{equation}
It follows from (\ref{compare:5}) and (\ref{compare:6}) that
\begin{multline*}
\int_{T_1}^{T_2} \langle z(t)-w(t),f(t) \rangle\, \mathrm dt
=\left\langle \chi_\omega u,\varphi(T_1)-\f{1}{\delta}\int_{T_1}^{T_1+\delta} \varphi(\tau)\,\mathrm d\tau \right\rangle 
=-\left\langle \chi_\omega u,\f{1}{\delta}\int_{T_1}^{T_1+\delta} \int_{T_1}^\tau \p_t \varphi\,\mathrm dt \,\mathrm d\tau\right\rangle,
\end{multline*}
which, together with (\ref{compare:7}), yields
\begin{equation*}
\left|\int_{T_1}^{T_2} \langle z(t)-w(t),f(t) \rangle\, \mathrm dt\right|\leq
\|\chi_\omega u\| \int_{T_1}^{T_1+\delta} \|\p_t \varphi\| \,\mathrm dt
\leq C(T_2,p) \delta^{\f{1}{p}} \|\chi_\omega u\|\|f\|_{L^q(T_1,T_2;L^2(\Omega))}.
\end{equation*}
This leads to the desired result and completes the proof.
\end{proof}

\subsubsection{Proof of the control error estimate}\label{sec232}
In this section, our objective is to establish \eqref{error_estim_control}.

Recalling that $u_n^*$ is defined by \eqref{error-2}, we denote by $y(u_n^*)$ and by $p(u_n^*)$ the solutions of
\begin{equation}\label{error-3}
\left\{\begin{array}{lll}
\p_t y(u_n^*)-\triangle y(u_n^*)=\chi_\omega u_n^*&\mbox{in}&\Omega\times
(0,T),\\
y(u_n^*)=0&\mbox{on}&\p \Omega\times (0,T),\\
y(u_n^*)(0)=y(u_n^*)(T)&\mbox{in}&\Omega
\end{array}\right.
\end{equation}
and
\begin{equation}\label{error-4}
\left\{\begin{array}{lll}
\p_t p(u_n^*)+\triangle p(u_n^*)=y(u_n^*)-y_d&\mbox{in}&\Omega\times
(0,T),\\
p(u_n^*)=0&\mbox{on}&\p \Omega\times (0,T),\\
p(u_n^*)(0)=p(u_n^*)(T)&\mbox{in}&\Omega.
\end{array}\right.
\end{equation}
The existence of these solutions follows from Section \ref{sec21}.
The proof goes in three steps.

\paragraph{Step 1.} We claim that
\begin{equation}\label{error-5}
\d{\int_{\tau_1}^T} \|u^*-u_n^*\|^2\,\mathrm dt = I_1+I_2
\end{equation}
with
$$
I_1 = \d{\sum_{i=2}^n\int_{\tau_{i-1}}^{\tau_i}} \langle \chi_\omega p^*-\chi_\omega p(u_n^*),u^*-u_n^*\rangle \,\mathrm dt ,
\qquad
I_2 = \d{\sum_{i=2}^n\int_{\tau_{i-1}}^{\tau_i}} \langle \chi_\omega p(u_n^*)-\chi_\omega p_n^*(\tau_{i-1}),u^*-u_n^*\rangle \,\mathrm dt,
$$
where $p^*$ and $p_n^*$ are given by \eqref{Maxi-2} and \eqref{Prin-2} respectively.

The claim follows from (\ref{Maxi-3}), (\ref{error-2}) and (\ref{Prin-3}), and from the fact that
$$
\int_{\tau_1}^{T} \|u^*-u_n^*\|^2\,\mathrm dt
=\d{\sum_{i=2}^n\int_{\tau_{i-1}}^{\tau_i}} \langle u^*-u_n^*,u^*-u_n^*\rangle \,\mathrm dt\\
=\d{\sum_{i=2}^n\int_{\tau_{i-1}}^{\tau_i}} \langle \chi_\omega p^*-\chi_\omega p_n^*(\tau_{i-1}),u^*-u_n^*\rangle \,\mathrm dt .
$$

\paragraph{Step 2.} We claim that
\begin{equation}\label{error-6}
I_1\leq  C(T)  h_n  \|y_d\|_{L^2(0,T;L^2(\Omega))}^2 .
\end{equation}
We first infer from \eqref{error-2} that
\begin{equation}\label{error-7}
\begin{array}{lll}
I_1&\equiv&\d{\sum_{i=2}^n\int_{\tau_{i-1}}^{\tau_i}} \langle \chi_\omega (p^*-p(u_n^*)),u^*-u_n^*\rangle \,\mathrm dt\\
&=&\d{\int_0^T} \langle p^*-p(u_n^*),\chi_\omega (u^*-u_n^*)\rangle \,\mathrm dt-
\d{\int_0^{\tau_1}} \langle p^*-p(u_n^*),\chi_\omega u^*\rangle \,\mathrm dt.
\end{array}
\end{equation}
Then, on one hand, by (\ref{Maxi-1}), (\ref{Maxi-2}), (\ref{error-3})
and (\ref{error-4}), we get that
\begin{equation}\label{error-8}
\left\{\begin{array}{lll}
\p_t (y^*-y(u_n^*))-\triangle (y^*-y(u_n^*))=\chi_\omega (u^*-u_n^*)&\mbox{in}&\Omega\times
(0,T),\\
y^*-y(u_n^*)=0&\mbox{on}&\p \Omega\times (0,T),\\
(y^*-y(u_n^*))(0)=(y^*-y(u_n^*))(T)&\mbox{in}&\Omega
\end{array}\right.
\end{equation}
and
\begin{equation}\label{error-9}
\left\{\begin{array}{lll}
\p_t (p^*-p(u_n^*))+\triangle (p^*-p(u_n^*))=y^*-y(u_n^*)&\mbox{in}&\Omega\times
(0,T),\\
p^*-p(u_n^*)=0&\mbox{on}&\p \Omega\times (0,T),\\
(p^*-p(u_n^*))(0)=(p^*-p(u_n^*))(T)&\mbox{in}&\Omega.
\end{array}\right.
\end{equation}
Multiplying the first equation of (\ref{error-8}) by $p^*-p(u_n^*)$ and integrating over $\Omega\times (0,T)$, by (\ref{error-8}) and (\ref{error-9}), we obtain that
\begin{equation}\label{error-10}
\int_0^T \langle p^*-p(u_n^*),\chi_\omega (u^*-u_n^*)\rangle \,\mathrm dt=-\int_0^T \|y^*-y(u_n^*)\|^2\,\mathrm dt\leq 0.
\end{equation}
On the other hand, since $(Y_n^*,U_n^*)$ is the optimal pair for the problem $\IOCPn$, we have $J_n(Y_n^*,U_n^*)\leq J_n(0,0)$.
Then by (\ref{Penalty}), (\ref{def_yn*})  and (\ref{error-2}), it follows that
\begin{equation}\label{error-11}
\int_0^T \|y_n^*\|^2\,\mathrm dt+\int_0^T \|u_n^*\|^2\,\mathrm dt\leq C \int_0^T \|y_d\|^2\,\mathrm dt .
\end{equation}
From (\ref{error-11}), (\ref{error-3}), (\ref{error-4}) and Lemma~\ref{Energy}, we infer that
\begin{equation}\label{error-12}
\|y(u_n^*)\|_{C([0,T];H_0^1(\Omega))}+\|p(u_n^*)\|_{C([0,T];H_0^1(\Omega))\cap H^1(0,T;L^2(\Omega))}\leq  C(T) \|y_d\|_{L^2(0,T;L^2(\Omega))}.
\end{equation}
 Since $(y^*,u^*)$ is the optimal pair for the problem $\OCP$, we have $J(y^*,u^*)\leq J(0,0)$. Then by (\ref{def_J}), we get
\begin{equation}\label{wang-3}
\int_0^T \|y^*\|^2\,\mathrm dt+\int_0^T \|u^*\|^2\,\mathrm dt\leq C\int_0^T \|y_d\|^2\,\mathrm dt,
\end{equation}
which, combined with (\ref{Maxi-2}) and Lemma~\ref{Energy}, implies that
\begin{equation}\label{wang-4}
\|p^*\|_{C([0,T];L^2(\Omega))}\leq C(T)\|y_d\|_{L^2(0,T;L^2(\Omega))}.
\end{equation}
 By (\ref{error-7}), (\ref{error-10}), (\ref{Maxi-3}), (\ref{error-12}) and (\ref{wang-4}), we get that
\begin{eqnarray*}
I_1&\leq& -\d{\int_0^{\tau_1}} \langle p^*-p(u_n^*),\chi_\omega p^*\rangle \,\mathrm dt\\
&\leq&
\tau_1\|p^*-p(u_n^*)\|_{C([0,T];L^2(\Omega))}\|p^*\|_{C([0,T];L^2(\Omega))}
\leq  C(T)  h_n \|y_d\|_{L^2(0,T;L^2(\Omega))}^2 ,
\end{eqnarray*}
and \eqref{error-6} follows.

\paragraph{Step 3.} We claim that
\begin{equation}\label{error-13}
I_2\leq \d{\f{1}{2}\int_{\tau_1}^{T}} \|u^*-u_n^*\|^2\,\mathrm dt+ C(T) h_n \|y_d\|_{L^2(0,T;L^2(\Omega))}^2 .
\end{equation}
We first note that
\begin{equation}\label{error-14}
\begin{array}{lll}
I_2&\equiv&\d{\sum_{i=2}^n\int_{\tau_{i-1}}^{\tau_i}} \langle \chi_\omega p(u_n^*)-\chi_\omega p_n^*(\tau_{i-1}),u^*-u_n^*\rangle \,\mathrm dt\\
&\leq&\d{\f{1}{2} \int_{\tau_1}^{T}} \|u^*-u_n^*\|^2\,\mathrm dt
+\d{\f{1}{2}\sum_{i=2}^n \int_{\tau_{i-1}}^{\tau_i}} \|p(u_n^*)-p_n^*(\tau_{i-1})\|^2\,\mathrm dt.
\end{array}
\end{equation}
Then, we proceed with three sub-steps.
\begin{itemize}
\item {\bf Sub-step 3.1.}
Let us prove that
\begin{equation}\label{error-15}
\sum_{i=2}^n \int_{\tau_{i-1}}^{\tau_i} \|p(u_n^*)-p_n^*(\tau_{i-1})\|^2\,\mathrm dt
\leq  C(T)   h_n^2 \|y_d\|_{L^2(0,T;L^2(\Omega))}^2 +  C(T)  \int_{0}^{T} \|y(u_n^*)-y_{n}^*\|^2\,\mathrm dt.
\end{equation}
By (\ref{error-12}), we have
\begin{multline}\label{error-16}
\d{\sum_{i=2}^n \int_{\tau_{i-1}}^{\tau_i}} \|p(u_n^*)-p(u_n^*)(\tau_{i-1})\|^2\,\mathrm dt
=\d{\sum_{i=2}^n\int_{\tau_{i-1}}^{\tau_i}}\Big\|\d{\int_{\tau_{i-1}}^t}\p_s p(u_n^*)\,\mathrm ds\Big\|^2\,\mathrm dt \\
\leq h_n^2\|\p_t p(u_n^*)\|_{L^2(0,T;L^2(\Omega))}^2\leq   C(T)  h_n^2 \|y_d\|_{L^2(0,T;L^2(\Omega))}^2 .
\end{multline}
Moreover, from (\ref{Prin-2}), (\ref{error-4}) and Lemma~\ref{Energy}, we obtain that
\begin{multline*}
\d{\sum_{i=2}^n \int_{\tau_{i-1}}^{\tau_i}} \|p(u_n^*)(\tau_{i-1})-p_n^*(\tau_{i-1})\|^2\,\mathrm dt
\leq T\|p(u_n^*)-p_n^*\|_{C([0,T];L^2(\Omega))}^2  
\leq  C(T)  \d{\int_0^T}\|y(u_n^*)-y_n^*\|^2\,\mathrm dt.
\end{multline*}
Combined with (\ref{error-16}), this gives (\ref{error-15}).

\item {\bf Sub-step 3.2.} We claim that
\begin{equation}\label{error-17}
\|e^{t\triangle}(y(u_n^*)(0)-y_{n}^*(0))\|_{L^p(0,T;L^2(\Omega))}\leq
 C(T,p)  h_n^{1/p} \|y_d\|_{L^2(0,T;L^2(\Omega))} ,
\end{equation}
for every $p\in [2,\infty)$.

Indeed, by (\ref{Prin-1}), (\ref{Prin-4}), (\ref{error-3}) and (\ref{error-2}), we have
\begin{equation*}
\left\{
\begin{array}{l}
y_n^*(0)=e^{T\triangle}y_n^*(0)+\d{\sum_{j=2}^n}  e^{(T-\tau_{j-1})\triangle} \chi_\omega u_{j-1,n}^*,\\
y(u_n^*)(0)=e^{T\triangle}y(u_n^*)(0)+\d{\sum_{j=2}^n\f{1}{h_n}\int_{\tau_{j-1}}^{\tau_j}}e^{(T-t)\triangle}\chi_\omega u_{j-1,n}^*\,\mathrm dt.
\end{array}\right.
\end{equation*}
Then
\begin{multline}\label{error-19}
y(u_n^*)(0)-y_n^*(0)=e^{T\triangle} (y(u_n^*)(0)-y_n^*(0))\\
+\d{\sum_{j=2}^n} \left(\f{1}{h_n}\int_{\tau_{j-1}}^{\tau_j}e^{(T-\tau)\triangle}\chi_\omega u_{j-1,n}^*\mathrm d\tau
-e^{(T-\tau_{j-1})\triangle}\chi_\omega u_{j-1,n}^*\right).
\end{multline}
It follows that
\begin{equation}
\begin{split}
&\|e^{t\triangle}(y(u_n^*)(0)-y_{n}^*(0))\|_{L^p(0,T;L^2(\Omega))}\nb\\
&\leq (1-e^{-\l_1 T})^{-1} \d{\sum_{j=2}^n}
\left\|\f{1}{h_n}\int_{\tau_{j-1}}^{\tau_j}e^{(T-\tau+t)\triangle}\chi_\omega u_{j-1,n}^*\mathrm d\tau
-e^{(T-\tau_{j-1}+t)\triangle}\chi_\omega u_{j-1,n}^*\right\|_{L^p(0,T;L^2(\Omega))}\nb\\
&\leq  C(T)  \d{\sum_{j=2}^n} \left\|\f{1}{h_n}\int_{\tau_{j-1}}^{\tau_j}e^{(t-\tau)\triangle}\chi_\omega u_{j-1,n}^*\mathrm d\tau
-e^{(t-\tau_{j-1})\triangle}\chi_\omega u_{j-1,n}^*\right\|_{L^p(\tau_j,2T;L^2(\Omega))},\label{error-20}
\end{split}
\end{equation}
where $-\l_1<0$ is the first eigenvalue of the Dirichlet Laplacian.
Moreover, by Lemma~\ref{compare:3} with $T_1=\tau_{j-1}, \delta=h_n, T_2=2T$ and $u=u_{j-1,n}^*$,
we get that
\begin{multline}\label{error-21}
\left\|\d{\f{1}{h_n}\int_{\tau_{j-1}}^t} e^{(t-\tau)\triangle}\chi_{(\tau_{j-1},\tau_j)}(\tau)\chi_\omega u_{j-1,n}^*\mathrm d\tau
-e^{(t-\tau_{j-1})\triangle}\chi_\omega u_{j-1,n}^*\right\|_{L^p(\tau_{j-1},2T;L^2(\Omega))}\\
\leq  C(T,p)   h_n^{1/p}\|\chi_\omega u_{j-1,n}^*\|,
\end{multline}
for every $j\in \{2,\ldots,n\}$.
Since
\begin{equation*}
\int_{\tau_{j-1}}^t e^{(t-\tau)\triangle}\chi_{(\tau_{j-1},\tau_j)}(\tau)\chi_\omega u_{j-1,n}^*\mathrm d\tau
=\int_{\tau_{j-1}}^{\tau_j} e^{(t-\tau)\triangle}\chi_\omega u_{j-1,n}^*\mathrm d\tau,\quad \forall\;t\in [\tau_j,2T],
\end{equation*}
by (\ref{error-20}) and (\ref{error-2}), we obtain that
\begin{equation}\label{error-22}
\begin{split}
\|e^{t\triangle}(y(u_n^*)(0)-y_{n}^*(0))\|_{L^p(0,T;L^2(\Omega))}
&\leq  C(T,p) \d{\sum_{j=2}^n} h_n^{1/p}\|u_{j-1,n}^*\| \\
&\leq  C(T,p)  h_n^{1/p}\left(\d{\sum_{j=2}^n} h_n\right)^{\f{1}{2}} \left(\d{\sum_{j=2}^n} \f{1}{h_n}\|u_{j-1,n}^*\|^2\right)^{\f{1}{2}}\\
&\leq  C(T,p)  h_n^{1/p}\|u^*_n\|_{L^2(0,T;L^2(\Omega))},
\end{split}
\end{equation}
which, combined with (\ref{error-11}), implies (\ref{error-17}).

\item {\bf Sub-step 3.3.} Let us prove that
\begin{equation}\label{error-23}
\|y(u_n^*)-y_n^*\|_{L^p(0,T;L^2(\Omega))}\leq
 C(T,p)   h_n^{1/p} \|y_d\|_{L^2(0,T;L^2(\Omega))} ,
\end{equation}
for every $p\in [2,+\infty)$.

Let $(z_{j,n})_{1\leq j\leq n}$ and $(w_{j,n})_{1\leq j\leq n}$ be
solutions of (\ref{compare:1}) and (\ref{compare:2}) respectively, with
$T_1=\tau_{j-1}, \delta=h_n, T_2=T$ and $u=u_{j-1,n}^*$. We set
\begin{equation*}
\widetilde{z}_{j,n}(t)=
\left\{
\begin{array}{ll}
0,&t\in (0,\tau_{j-1}],\\
z_{j,n}(t),&t\in (\tau_{j-1},T),
\end{array}\right.\;\;\mbox{and}\quad
\widetilde{w}_{j,n}(t)=
\left\{
\begin{array}{ll}
0,&t\in (0,\tau_{j-1}],\\
w_{j,n}(t),&t\in (\tau_{j-1},T).
\end{array}\right.
\end{equation*}
By (\ref{Prin-1}), (\ref{Prin-4}), (\ref{error-2}) and (\ref{error-3}),
we have
\begin{equation}\label{error-25}
\begin{split}
y(u_n^*)(t)-y_n^*(t)
&=e^{t\triangle}(y(u_n^*)(0)-y_n^*(0))\\
&+\d{\sum_{j=1}^i} \left(\d{\int_{\tau_{j-1}}^{\tau_{j}}} \chi_{(\tau_{j-1},t)}(s)e^{(t-s)\triangle}\chi_\omega u_n^*(s)\,\mathrm ds
-e^{(t-\tau_{j-1})\triangle}\chi_\omega u_{j-1,n}^*\right)\\
&=e^{t\triangle}(y(u_n^*)(0)-y_n^*(0))+\d{\sum_{j=1}^i}(z_{j,n}(t)-w_{j,n}(t)) ,
\end{split}
\end{equation}
for every $t\in (\tau_{i-1},\tau_i]$ and every $i\in \{1,\ldots,n\}$.
Then
\begin{equation}\label{error-26}
y(u_n^*)(t)-y_n^*(t)\\
=e^{t\triangle}(y(u_n^*)(0)-y_n^*(0))
+\d{\sum_{j=1}^n} \chi_{(\tau_{j-1},T)}(t)\left(\widetilde{z}_{j,n}(t)-\widetilde{w}_{j,n}(t)\right),
\end{equation}
for every $t\in (0,T)$.
Indeed, given any $t\in (0,T)$, let $i_0\in \{1,\ldots,n\}$ be such that $t\in (\tau_{i_0-1},\tau_{i_0}]$.
It follows from (\ref{error-25}) that
\begin{equation*}
\begin{split}
y(u_n^*)(t)-y_n^*(t)&=e^{t\triangle}(y(u_n^*)(0)-y_n^*(0))+\d{\sum_{j=1}^{i_0}} (z_{j,n}(t)-w_{j,n}(t))\\
&=e^{t\triangle}(y(u_n^*)(0)-y_n^*(0))+\d{\sum_{j=1}^{i_0}} \left(\widetilde{z}_{j,n}(t)-\widetilde{w}_{j,n}(t)\right)\\
&=e^{t\triangle}(y(u_n^*)(0)-y_n^*(0))+\d{\sum_{j=1}^{i_0}} \chi_{(\tau_{j-1},T)}(t)
\left(\widetilde{z}_{j,n}(t)-\widetilde{w}_{j,n}(t)\right),
\end{split}
\end{equation*}
which yields (\ref{error-26}). By (\ref{error-26}), (\ref{error-17}) and Lemma~\ref{compare:3},
we obtain that
\begin{equation*}
\begin{split}
&\|y(u_n^*)-y_{n}^*\|_{L^p(0,T;L^2(\Omega))}\\
&\leq \|e^{t\triangle} (y(u_n^*)(0)-y_{n}^*(0))\|_{L^p(0,T;L^2(\Omega))}
+\d{\sum_{j=1}^n} \|z_{j,n}-w_{j,n}\|_{L^p(\tau_{j-1},T;L^2(\Omega))}\\
&\leq  C(T,p) \left(h_n^{1/p} \|y_d\|_{L^2(0,T;L^2(\Omega))} +\d{\sum_{j=1}^n} h_n^{1/p} \|u_{j-1,n}^*\|\right).
\end{split}
\end{equation*}
Using (\ref{error-11}) and the same arguments as in (\ref{error-22}), we obtain (\ref{error-23}).
\end{itemize}
Step 3 follows immediately from (\ref{error-14}), (\ref{error-15}) and (\ref{error-23}).\\
Finally, the theorem follows from (\ref{error-5}), (\ref{error-6}), (\ref{error-13}), (\ref{error-2}), (\ref{Maxi-3})
 and (\ref{wang-4}).

\subsubsection{Proof of the state  and cost functional error estimates} 
In this section, our objective is to establish \eqref{wang-1}  and (\ref{Lp-estimate:1}). 

We start with the case $2\leq p<+\infty$.
By the triangular inequality, we have
\begin{equation}\label{Lp-estimate:2}
\|y^*-y_n^*\|_{L^p(0,T;L^2(\Omega))}\leq \|y^*-y(u_n^*)\|_{L^p(0,T;L^2(\Omega))}+\|y(u_n^*)-y_n^*\|_{L^p(0,T;L^2(\Omega))}.
\end{equation}
We infer from \eqref{Maxi-1}, \eqref{error-3} and from Lemma \ref{Energy} that
\begin{equation}\label{Lp-estimate:3}
\|y^*-y(u_n^*)\|_{L^p(0,T;L^2(\Omega))} \leq T^{\f{1}{p}}\|y^*-y(u_n^*)\|_{C([0,T];L^2(\Omega))}\leq C(T,p)\|u^*-u_n^*\|_{L^2(0,T;L^2(\Omega))}.
\end{equation}
Since $p\geq 2$, it follows from \eqref{Lp-estimate:3} and from the first part of Theorem \ref{Theorem-Error} that
\begin{equation*}
\|y^*-y(u_n^*)\|_{L^p(0,T;L^2(\Omega))}\leq C(T,p) \sqrt{h_n} \|y_d\|_{L^2(0,T;L^2(\Omega))} 
\leq C(T,p) h_n^{1/p} \|y_d\|_{L^2(0,T;L^2(\Omega))} ,
\end{equation*}
which, combined with \eqref{Lp-estimate:2} and \eqref{error-23}, gives \eqref{Lp-estimate:1}.

 Finally, (\ref{wang-1}) follows from (\ref{def_J}), (\ref{Penalty}), (\ref{def_yn*}), (\ref{error-2}), (\ref{error_estim_control}), (\ref{Lp-estimate:1}), (\ref{error-11}) and (\ref{wang-3}).

\subsection{Proof of Theorem \ref{infty-estimate}}\label{sec_proof_infty-estimate}
\subsubsection{A general result in measure theory}
\begin{lemma}\label{Measure}
Let $\omega$ be a measurable subset of $\Omega$ having a $C^2$ boundary $\p\omega$.
For $\e>0$, we define
\begin{equation}\label{measurable set}
\omega_\e=\{x\in \mathbb{R}^N\ \mid\ d(x,\p\omega)\leq \e\},
\end{equation}
where $d(x,\p\omega)= \inf\{ |x-y|\ \mid \ y\in \p\omega\}$.
There exists $\mu>0$ such that
\begin{equation}\label{measure-1}
|\omega_\e|=\int_0^{\e}|\p\omega_\eta|
\,\mathrm d\eta \leq 2(1+\e/\mu)^{N-1}|\p\omega|\e,
\end{equation}
for every $\e\in (0,\mu)$.
\end{lemma}

In \eqref{measure-1}, without ambiguity, $|\omega_\e|$ designates the Lebesgue measure of $\omega_\e$, and $|\p\omega_\eta|=\mathcal H^{N-1}(\p\omega_\eta)$ designates the $(N-1)$-Hausdorff measure of $\p\omega_\eta$.

We give a proof of this result for completeness, borrowing arguments from \cite{Gilbarg}.

\begin{remark}
In the proof below, the assumption $\p\omega\in C^2$ is required.
For the general case, whether (\ref{measure-1}) holds or not seems to be open.
\end{remark}

\begin{proof}
For every $y\in \p\omega$, let $\nu(y)$ and $\Gamma(y)$ respectively denote the unit inner normal to $\p\omega$ at $y$ and the tangent hyperplane to $\p\omega$ at $y$. The curvatures of $\p\omega$ at a fixed point $y_0\in \p\omega$ are determined as follows. By a rotation of coordinates, we assume that the $x_N$ coordinate axis lies in the
direction $\nu(y_0)$. In some neighborhood $\mathcal{N}(y_0)$ of $y_0$, we have $ \mathcal{N}(y_0)\cap\p\omega=\{x_N=\varphi(x^\prime)\}$,
 where $x^\prime=(x_1,\dots,x_{N-1}), \varphi\in C^2(\Gamma(y_0)\cap \mathcal{N}(y_0))$ and $D\varphi(y_0^\prime)=0$.
The eigenvalues $\kappa_1,\ldots,\kappa_{N-1}$ of the Hessian matrix $D^2\varphi(y_0^\prime)$ are the principal curvatures of $\p\omega$ at $y_0$ and the corresponding eigenvectors are the principal directions of $\p\omega$ at $y_0$. By an additional rotation of coordinates, we assume that the $x_1,\ldots,x_{N-1}$ axes lie
along principal directions corresponding to $\kappa_1,\ldots,\kappa_{N-1}$ at $y_0$. Such a coordinate system is said to be a principal coordinate system at $y_0$. The Hessian matrix $D^2\varphi(y_0^\prime)$ with respect to the principal
coordinate system at $y_0$ described above is given by $D^2\varphi(y_0^\prime)=\mbox{diag}(\kappa_1,\ldots,\kappa_{N-1})$.
The unit inner normal vector $\nu(y)$ at the point $y=(y^\prime,\varphi(y^\prime))\in \mathcal{N}(y_0)\cap \p\omega$
is given by
\begin{equation*}
\nu_i(y)=-\f{D_i \varphi(y^\prime)}{\sqrt{1+|D\varphi(y^\prime)|^2}},\quad 1\leq i\leq N-1,
\qquad\nu_N(y)=\f{1}{\sqrt{1+|D\varphi(y^\prime)|^2}}.
\end{equation*}
Hence, with respect to the principal coordinate system at $y_0$, we have
\begin{equation}\label{measure-2}
D_j v_i(y_0)=-\kappa_i \delta_{ij},\quad i, j=1,\ldots,N-1.
\end{equation}
Since $\p\omega$ is $C^2$, $\p\omega$ satisfies a uniform interior and exterior sphere condition,
i.e., at each
point $y_0\in \p\omega$, there exist two balls $B_1$ and $B_2$ depending on $y_0$ such that
$\o B_1\cap (\mathbb{R}^N-\omega)=\{y_0\}$ and $\o B_2\cap \o \omega=\{y_0\}$,
and the radii of the balls $B_1$ and $B_2$ are bounded below by a positive constant denoted by $\mu$.
It is easy to show that $\mu^{-1}$ bounds the principal curvatures of $\p\omega$.

The rest of the proof goes in two steps.

\paragraph{Step 1.} Let us prove that $\omega_\e$ $(0<\e<\mu)$ has a $C^1$-smooth manifold structure.

Given any point $x$ such that $d(x,\p\omega)<\mu$, there exists a unique point $y=y(x)\in \p\omega$ satisfying $|x-y|=d(x,\p\omega)$. We have $x=y+\nu(y) d(x,\p\omega) $ if $x\in\omega$ and $x=y-\nu(y) d(x,\p\omega)$ if $x\not\in\omega$.
Now we give a construction of a $C^1$-smooth manifold structure on $\omega_\e$.
For this purpose, we fix a $y_0\in \p\omega$ and we define the $C^1$ map $\Phi_0$ from
$\mathcal{U}= (\Gamma(y_0)\cap \mathcal{N}(y_0))\times (-\mu,\mu)$
to $\mathbb{R}^N$ by
\begin{equation}\label{measure-4}
\Phi_0(y^\prime,d)= y+\nu(y)d,\quad \forall\;
(y^\prime,d)\in (\Gamma(y_0)\cap \mathcal{N}(y_0))\times (-\mu,\mu),
\end{equation}
where $y=(y^\prime,\varphi(y^\prime))$. By (\ref{measure-2}), the Jacobian matrix of $\Phi_0$
at $(y_0^\prime,d)$ is $D\Phi_0(y_0^\prime,d)=\mbox{diag}(1-\kappa_1 d,\ldots,1-\kappa_{N-1}d,1)$, and hence $\det D\Phi_0(y_0^\prime,d)=(1-\kappa_1 d)\cdots(1-\kappa_{N-1}d)\neq 0$, for every $d\in (-\mu,\mu)$.
It follows from the inverse function theorem that $\Phi_0$ is a local $C^1$-diffeomorphism in a neighborhood of any point of the line $\{y_0^\prime\}\times(-\mu,\mu)$.
Then by compactness of $[-\e,\e]$,
we can choose $\mathcal U_0=B_0\times[-\e,\e]$,
with $B_0$ an open ball in $\Gamma(y_0)\cap \mathcal N(y_0)$,
so that $\Phi_0$ is a $C^1$-diffeomorphism from $\mathcal U_0$ to
$\Phi_0(\mathcal U_0)$. This shows that $(\Phi_0(\mathcal U_0), \Phi_0^{-1})$ is
a coordinate chart centered at
$y_0$ in the topological space $\omega_\e$.

We carry on the above process for each $y\in\p\omega$ and we
define an atlas $\{(V_\alpha, \Phi_\alpha^{-1})\}$ for $\omega_\e$,
where $V_\alpha$ is an open neighborhood of $y_\alpha\in\p\omega$,
$\Phi_\alpha^{-1}(V_\alpha)=\mathcal U_\alpha= B_\alpha\times[-\e,\e]$
and $B_\alpha$ is an open ball in $\Gamma(y_\alpha)\cap \mathcal N(y_\alpha)$.
By the definition of $\Phi_\alpha$ (similar to (\ref{measure-4})), one can check that any two charts in
$\{(V_\alpha,\Phi_\alpha^{-1})\}$ are $C^1$-smoothly compatible one with each other.
Hence $\{(V_\alpha, \Phi_\alpha^{-1})\}$ is a $C^1$ atlas for $\omega_\e$.
This atlas induces a $C^1$ structure on $\omega_\e$.

\paragraph{Step 2.} Let us establish (\ref{measure-1}).

By \cite[Lemma 14.16, page 355]{Gilbarg}, we have $d(\cdot,\p\omega)\in C^2(\omega_\e)$ and $|\nabla d(\cdot,\p\omega)|=1$ in $\omega_\e$,
which, combined with the Coarea Formula (see, e.g., \cite{Evans}) applied to $f=d(\cdot,\p\omega)$, gives
\begin{multline}\label{measure-5}
|\omega_\e|=\d{\int_{\omega_\e}} |\nabla d(x,\p\omega)|\,\mathrm dx
=\d{\int_0^\e } \mathcal H^{N-1}(\{d(\cdot,\p\omega)=\eta\})\,\mathrm d\eta 
=\d{\int_0^\e} |\p\omega_\eta|\,\mathrm d\eta=\d{\int_0^\e} (|\p\omega_\eta^+|+|\p\omega_\eta^-|)\,\mathrm d\eta,
\end{multline}
where $\p\omega_\delta^+$ and $\p\omega_\delta^-$ are the inner and outer parts
(with respect to $\omega$) of $\p\omega_\delta$ for each $\delta\in(0,\e)$, respectively.
Now, given any $\eta\in (0, \e)$, to compute $|\p\omega_\eta^+|$ and $|\p\omega_\eta^-|$,
we  define $\psi_\eta: \p\omega\ra \p\omega_\eta^+$ by
$\psi_\eta(y_\alpha)=\Phi_\alpha\circ\tau_\eta\circ\Phi_\alpha^{-1}(y_\alpha)$, for every $y_\alpha\in \p\omega$,
where $\tau_\eta$ is the mapping given by $\tau_\eta(z,0)= (z,\eta)$ for every $z\in\mathbb{R}^{N-1}$.
From Step 1, we take two arbitrary coordinate charts $\{(V_\beta, \Phi_\beta^{-1})\}$ and $\{(V_\gamma, \Phi_\gamma^{-1})\}$,
where $V_\beta$ and $V_\gamma$ are  open neighborhood of $y_\beta$ and $y_\gamma$ ( $y_\beta,y_\gamma\in \partial\omega$), respectively.
Then by the definitions of $\Phi_\beta$ and $\Phi_\gamma$ (similar to (\ref{measure-4})),
one can check that
\begin{equation}\label{compactible-psi}
 \Phi_\beta\circ\tau_\eta\circ\Phi_\beta^{-1}(y)=\Phi_\gamma\circ\tau_\eta\circ\Phi_\gamma^{-1}(y),
 \quad \forall\,y\in V_\beta \cap V_\gamma\cap \partial\omega.
\end{equation}
We recall from Step 1 that each $\Phi_\alpha^{-1}$ is $C^1$ diffeomorphic from $V_\alpha$ to
$\mathcal U_\alpha=\Phi_\alpha(V_\alpha)$. Therefore, by (\ref{compactible-psi}),
$\psi_\eta$ is $C^1$ diffeomorphic from $\p\omega$ onto $\p\omega_\eta^+$ and
$$
\det (\Phi_\alpha\circ\tau_\eta\circ\Phi_\alpha^{-1})(y_\alpha)=(1-\kappa_1(y_\alpha)\eta)\cdots(1-\kappa_{N-1}(y_\alpha)\eta)
\in((1-\e\mu^{-1})^{N-1},(1+\e\mu^{-1})^{N-1}),
$$
for every $y_\alpha\in \p\omega$.
This, together with the definition of $\psi_\eta$ and (\ref{compactible-psi}),  implies  that
\begin{equation}\label{measure-6}
|\p\omega_\eta^+|=\int_{\p\omega} \mbox{det}[\psi_\eta](x)\,\mathrm d\sigma
\leq (1+\eta\mu^{-1})^{N-1} |\p\omega|,\quad \forall\;\eta\in (0,\e).
\end{equation}
Similarly, we have $|\p\omega_\eta^-|\leq (1+\eta\mu^{-1})^{N-1} |\p\omega|$, for every $\eta\in (0,\e)$.
Then, (\ref{measure-1}) follows from the latter inequality, from (\ref{measure-6}) and (\ref{measure-5}).

This completes the proof.
\end{proof}

\subsubsection{Smooth regularizations of characteristic functions}
We define the $C^\infty$ function $\chi_\omega^\e: \mathbb{R}^N\ra \mathbb{R}$ by
\begin{equation}\label{moll-1}
\chi_\omega^\e(x)= \int_\omega \eta_\e(x-y)\chi_\omega(y)\,\mathrm dy ,
\end{equation}
for every $x\in\mathbb{R}^N$, where
\begin{equation}\label{moll-2}
\eta_\e(x)= \f{1}{\e^N}\eta\left(\f{x}{\e}\right),
\qquad
\eta(x)= \left\{
\begin{array}{ll}
c \exp\left(\f{1}{|x|^2-1}\right)&\mbox{if}\;\;|x|<1,\\
0&\mbox{if}\;\;|x|\geq 1,
\end{array}
\right.
\end{equation}
with $c>0$ such that $\int_{\mathbb{R}^N}\eta(x)\,\mathrm dx=1$.

\begin{lemma}\label{moll-4}
 Let $\mu$ be as in Lemma \ref{Measure} and let $\e\in (0,\mu)$. For every $p\in[1,+\infty]$, we have
\begin{equation}\label{moll-5}
\|\nabla \chi_\omega^\e\|_{L^p(\Omega)}\leq C\e^{-1+\f{1}{p}}
\quad\mbox{and}\quad\|\chi_\omega^\e-\chi_\omega\|_{L^p(\Omega)}\leq C\e^{\f{1}{p}}.
\end{equation}
\end{lemma}

Here and throughout the proof, $C$ is a generic positive constant independent of $p$ and $\e$.

\begin{proof}
Note that the case $p=+\infty$ follows by passing to the limit. Therefore it suffices to prove (\ref{moll-5}) for $1\leq p<+\infty$.
We set $\omega_\e^1=\{x\in \omega: d(x,\p\omega)>\e\}$ and $\omega_\e^2= \{x\not\in\omega: d(x,\p\omega)>\e\}$.
Then $\omega_\e^1$ and $\omega_\e^2$ are open subsets of $\mathbb{R}^N$ such that
\begin{equation}\label{moll-5:1}
\omega_\e^1\cup \omega_\e^2\cup \omega_\e=\mathbb{R}^N,
\end{equation}
where $\omega_\e$ is defined by (\ref{measurable set}).

On the one hand, by (\ref{moll-1}) and (\ref{moll-2}), we get that
\begin{equation}\label{moll-5:2}
\chi_\omega^\e(x)=\int_{\mathbb{R}^N} \f{1}{\e^N}\eta\left(\f{x-y}{\e}\right)\chi_\omega(y)\,\mathrm dy,
\end{equation}
which, combined with (\ref{moll-2}), yields
\begin{multline}\label{moll-6}
\nabla\chi_\omega^\e(x)=\d{\int_{\mathbb{R}^N}\f{1}{\e^N}} \d{\f{1}{\e}}\chi_\omega(y)\nabla\eta\left(\d{\f{x-y}{\e}}\right)
\,\mathrm dy
=\d{\f{1}{\e}\int_{\mathbb{R}^N}} \chi_\omega(x-\e y)\nabla\eta(y)\,\mathrm dy\\
=\d{\f{1}{\e}\int_{\{y\in \mathbb{R}^N: |y|\leq 1\}}} \chi_\omega(x-\e y)\nabla\eta(y)\,\mathrm dy.
\end{multline}
On the other hand, by (\ref{moll-5:2}), we have
$$
\chi_\omega^\e(x)=\d{\int_{\mathbb{R}^N}} \eta(y)\chi_\omega(x-\e y)\,\mathrm dy
=\d{\int_{\{y\in \mathbb{R}^N: |y|<1\}}} \eta(y)\chi_\omega(x-\e y)\,\mathrm dy.
$$
This implies that
\begin{equation}\label{moll-7}
\chi_\omega^\e(x)=\left\{
\begin{array}{lll}
1&\mbox{if}&x\in \omega_\e^1,\\
0&\mbox{if}&x\in \omega_\e^2.
\end{array}
\right.
\end{equation}
It follows from (\ref{moll-5:1}), (\ref{moll-6}) and (\ref{moll-7}) that
$\|\nabla \chi_\omega^\e\|_{L^p(\Omega)}^p\leq \int_{\omega_\e} |\nabla \chi_\omega^\e(x)|^p\,\mathrm dx\leq (C\e^{-1})^p |\omega_\e|$,
which, combined with Lemma \ref{Measure}, yields $\|\nabla \chi_\omega^\e\|_{L^p(\Omega)}\leq C\e^{-1+\f{1}{p}}$.

Besides, by (\ref{moll-7}) and (\ref{moll-5:1}), we have
\begin{equation*}
|\chi_\omega^\e(x)-\chi_\omega(x)|\leq
\left\{
\begin{array}{lll}
0&\mbox{if}&x\in \omega_\e^1\cup \omega_\e^2,\\
1&\mbox{if}&x\in \omega_\e.
\end{array}\right.
\end{equation*}
This, together with Lemma \ref{Measure}, implies that $\|\chi_\omega^\e-\chi_\omega\|_{L^p(\Omega)}\leq |\omega_\e|^{\f{1}{p}}\leq C\e^{\f{1}{p}}$.
This completes the proof.
\end{proof}

\subsubsection{A useful estimate}\label{sec_usefulestimate}
The following estimate for a linear heat equation is not standard.

\begin{lemma}\label{intuitive}
Let $\omega\subset\Omega$ be a subset having a $C^2$ boundary. Let $p\in (1,+\infty)$, let $T_1$ and $T_2$ be two nonnegative real numbers such that $T_1<T_2$, and let $z_0\in H_0^1(\Omega)$. Let $z$ be the solution of
\begin{equation}\label{intuitive-1}
\left\{
\begin{array}{lll}
\p_t z-\triangle z=0&\mbox{in}&\Omega\times (T_1,T_2),\\
z=0&\mbox{on}&\p\Omega\times (T_1,T_2),\\
z(T_1)=\chi_\omega z_0&\mbox{in}&\Omega.
\end{array}\right.
\end{equation}
If $\omega=\Omega$, then
\begin{equation}\label{intuitive-3}
\|z(s)-z(T_1)\|\leq (s-T_1)^{\f{1}{2}}\|z_0\|_{H_0^1(\Omega)},
\end{equation}
for every $s\in [T_1,T_2]$.

If $\omega\not=\Omega$, then
\begin{equation}\label{intuitive-2}
\|z(s)-z(T_1)\|\leq\left\{
\begin{array}{rll}
C(T_2)(s-T_1)^{\f{1}{2N}}\|z_0\|_{H_0^1(\Omega)}&\mbox{if}&N\geq 3,\\
C(T_2,p)(s-T_1)^{\f{1}{4p}}\|z_0\|_{H_0^1(\Omega)}&\mbox{if}&N=2,\\
C(T_2)(s-T_1)^{\f{1}{4}}\|z_0\|_{H_0^1(\Omega)}&\mbox{if}&N=1,
\end{array}\right.
\end{equation}
for every $s\in [T_1,T_2]$, for some constants $C(T_2)>0$ and $C(T_2,p)$ not depending on $z_0$.
\end{lemma}

\begin{proof}
Since the proof of \eqref{intuitive-3} is similar to obtain but simpler than the one of \eqref{intuitive-2}, we assume that we are in the (more difficult) case where $\omega\not=\Omega$.
Let $\mu$ be as in Lemma \ref{Measure} and let $s\in (T_1,T_2]$. We set
\begin{equation}\label{intuitive-4}
c_0= \f{\min\{\mu,1\}}{2\max\{\sqrt{T_2},1\}}\quad\mbox{and}\quad \e= c_0\sqrt{s-T_1}.
\end{equation}
Note that $\e<\min\{\mu,1\}$.
By (\ref{intuitive-1}),  we have
\begin{equation}\label{intuitive-5}
z(s)=e^{(s-T_1)\triangle}(\chi_{\omega}-\chi_{\omega}^\e)z_0+e^{(s-T_1)\triangle}\chi_{\omega}^\e z_0
= z_1(s)+z_2(s),
\end{equation}
for every $s\in [T_1,T_2]$.
We have
$$
\|z_2(s)-z_2(T_1)\|^2=\d{\int_\Omega} \Big|\d{\int_{T_1}^s} \p_t z_2\, \mathrm dt\Big|^2\mathrm dx
\leq (s-T_1)\d{\int_{T_1}^{T_2}} \|\p_t z_2\|^2\, \mathrm dt,
$$
for every $s\in [T_1,T_2]$. By definition, $z_2$ is the unique solution of the Dirichlet heat equation with initial condition $z_2(T_1)=\chi_{\omega}^\e z_0$. By integration by parts, we have
\begin{equation*}
\begin{split}
\d{\int_{T_1}^{T_2}} \|\p_t z_2\|^2\, \mathrm dt
&   =  \int_{T_1}^{T_2}\int_\Omega (\p_t z_2(t,x))^2\, \mathrm dx\, \mathrm dt =  \int_{T_1}^{T_2}\int_\Omega \p_t z_2(t,x)\cdot \triangle z_2(t,x)\, \mathrm dx\, \mathrm dt\\
&   = \frac{1}{2}\int_\Omega |\nabla z_2(T_1,x)|^2 \mathrm dx - \frac{1}{2}\int_\Omega |\nabla z_2(T_2,x)|^2 \mathrm dx  \leq \Vert\nabla z_2(T_1)\Vert^2 ,
\end{split}
\end{equation*}
and therefore we get that
$$
\|z_2(s)-z_2(T_1)\|^2 \leq (s-T_1) \|\nabla z_2(T_1)\|^2,
$$
for every $s\in [T_1,T_2]$. It follows from (\ref{intuitive-5}) that
\begin{equation}\label{intuitive-7}
\|z(s)-z(T_1)\|\leq\|z_1(s)-z_1(T_1)\|+\|z_2(s)-z_2(T_1)\|
\leq2\|(\chi_{\omega}-\chi_{\omega}^\e)z_0\|+\sqrt{s-T_1} \|\nabla (\chi_{\omega}^\e z_0)\|.
\end{equation}
If $N\geq3$, then, using the H\"{o}lder inequality, the Sobolev embedding $H^1_0(\Omega)\hookrightarrow L^{\f{2N}{N-2}}(\Omega)$,
Lemma~\ref{moll-4} and (\ref{intuitive-4}), we obtain that
\begin{equation*}
\|(\chi_{\omega}-\chi_{\omega}^\e)z_0\| \leq
\|z_0\|_{L^{\f{2N}{N-2}}(\Omega)}\|\chi_\omega-\chi_{\omega}^\e\|_{L^N(\Omega)}
\leq C\|z_0\|_{H_0^1(\Omega)} \e^{\f{1}{N}} ,
\end{equation*}
and
\begin{equation*}
 \|\nabla (\chi_{\omega}^\e z_0)\|\leq \|z_0\|_{H_0^1(\Omega)}
 +\|z_0\|_{L^{\f{2N}{N-2}}(\Omega)}\|\nabla\chi_{\omega}^\e\|_{L^N(\Omega)}
 \leq C\|z_0\|_{H_0^1(\Omega)} (1+\e^{\f{1}{N}-1}).
\end{equation*}
These estimates, together with (\ref{intuitive-7}), imply that
\begin{equation}\label{intuitive-8}
\|z(s)-z(T_1)\|\leq C\|z_0\|_{H_0^1(\Omega)}\Big(\e^{\f{1}{N}}+\sqrt{s-T_1}\e^{\f{1}{N}-1}\Big).
\end{equation}
From (\ref{intuitive-8}) and (\ref{intuitive-4}) it follows that (\ref{intuitive-2}) holds.

If $N=2$, then, similarly,
\begin{equation*}
\|(\chi_{\omega}-\chi_{\omega}^\e)z_0\|\leq
\|z_0\|_{L^{\f{2p}{p-1}}(\Omega)}\|\chi_\omega-\chi_{\omega}^\e\|_{L^{2p}(\Omega)}
\leq C(p)\|z_0\|_{H_0^1(\Omega)} \e^{\f{1}{2p}} ,
\end{equation*}
\begin{equation*}
\|\nabla (\chi_{\omega}^\e z_0)\|\leq \|z_0\|_{H_0^1(\Omega)}+
\|z_0\|_{L^{\f{2p}{p-1}}(\Omega)}\|\nabla\chi_{\omega}^\e\|_{L^{2p}(\Omega)}
\leq C(p)\|z_0\|_{H_0^1(\Omega)} (1+\e^{\f{1}{2p}-1}),
\end{equation*}
and using (\ref{intuitive-7}) we infer that
\begin{equation}\label{intuitive-9}
\|z(s)-z(T_1)\|\leq C(p)\|z_0\|_{H_0^1(\Omega)}
\Big(\e^{\f{1}{2p}}+\sqrt{s-T_1}\e^{\f{1}{2p}-1}\Big).
\end{equation}
It follows from (\ref{intuitive-9}) and (\ref{intuitive-4})  that (\ref{intuitive-2}) holds.

If $N=1$, then
\begin{equation*}
\|(\chi_{\omega}-\chi_{\omega}^\e)z_0\| \leq
\|z_0\|_{C(\o{\Omega})}\|\chi_\omega-\chi_{\omega}^\e\|
\leq C\|z_0\|_{H_0^1(\Omega)} \e^{\f{1}{2}} ,
\end{equation*}
\begin{equation*}
\|\nabla (\chi_{\omega}^\e z_0)\|
\leq \|z_0\|_{H_0^1(\Omega)}+\|z_0\|_{C(\o{\Omega})}\|\nabla\chi_{\omega}^\e\|
\leq  C\|z_0\|_{H_0^1(\Omega)} (1+\e^{-\f{1}{2}}),
\end{equation*}
which, combined with (\ref{intuitive-7}), imply that
\begin{equation}\label{intuitive-10}
\|z(s)-z(T_1)\|\leq C\|z_0\|_{H_0^1(\Omega)}\Big(\e^{\f{1}{2}}+\sqrt{s-T_1}\e^{-\f{1}{2}}\Big).
\end{equation}
By (\ref{intuitive-10}) and (\ref{intuitive-4}), we obtain (\ref{intuitive-2}).
The proof is complete.
\end{proof}

\subsubsection{Proof of the state error estimates}
We prove (\ref{infty-estimate:1}) only when $N\geq 3$, the other cases being similar.
Let $u^*$ and $U_n^*$ be the optimal controls solutions of $\OCP$ and $\IOCPn$ respectively, where $U_n^*=(u_{1,n}^*,u_{2,n}^*,\ldots,u_{n-1,n}^*)\in (L^2(\Omega))^{n-1}$.
Let $u_n^*$ be given by (\ref{error-2}). We have
\begin{equation}\label{infty-estimate:3}
\|y^*-y_n^*\|_{L^\infty(0,T;L^2(\Omega))}\leq\|y^*-y(u_n^*)\|_{L^\infty(0,T;L^2(\Omega))}+\|y(u_n^*)-y_n^*\|_{L^\infty(0,T;L^2(\Omega))}.
\end{equation}
By (\ref{Maxi-1}), (\ref{error-3}) and Lemma~\ref{Energy}, we infer that
\begin{equation}\label{infty-estimate:4}
\|y^*-y(u_n^*)\|_{C([0,T];L^2(\Omega))}\leq   C(T)  \|u^*-u_n^*\|_{L^2(0,T;L^2(\Omega))}.
\end{equation}
Besides, we claim that
\begin{equation}\label{infty-estimate:5}
 \|y(u_n^*)-y_n^*\|_{L^\infty(0,T;L^2(\Omega))}\leq  C(T)  h_n^{1/2N} \|y_d\|_{L^2(0,T;L^2(\Omega))} .
\end{equation}
Then  (\ref{infty-estimate:1}) follows from (\ref{infty-estimate:3}), (\ref{infty-estimate:4}), (\ref{infty-estimate:5}) and Theorem~\ref{Theorem-Error}.

Let us prove (\ref{infty-estimate:5}).
On the one hand, by  (\ref{error-19}) and (\ref{Prin-3}), we have
\begin{equation}\label{infty-estimate:6}
\begin{split}
\|y_{n}^*(0)-y(u_n^*)(0)\|
&\leq\left\|(I-e^{T\triangle})^{-1} \d{\sum_{j=2}^n} \left(\f{1}{h_n}\int_{\tau_{j-1}}^{\tau_j}e^{(T-s)\triangle}\chi_\omega u_{j-1,n}^*\,\mathrm ds-e^{(T-\tau_{j-1})\triangle}\chi_\omega u_{j-1,n}^*\right)\right\|\\
&\leq  C(T)  \left\|\d{\sum_{j=2}^n}  \int_{\tau_{j-1}}^{\tau_j} e^{(T-s)\triangle}
\left( I-e^{(s-\tau_{j-1})\triangle}\right) \chi_\omega p^*_n(\tau_{j-1}) \, \mathrm ds\right\|\\
&\leq  C(T)  \d{\sum_{j=2}^n} \int_{\tau_{j-1}}^{\tau_j}
\left\| \left(I-e^{(s-\tau_{j-1})\triangle}\right) \chi_\omega p^*_n(\tau_{j-1})\right\| \,\mathrm ds.
\end{split}
\end{equation}
On the other hand, by (\ref{error-25}), (\ref{error-2}) and (\ref{Prin-3}), we infer that:
\begin{itemize}
\item for every $t\in [0,\tau_1]$,
\begin{equation}\label{infty-estimate:7}
y(u_n^*)(t)-y_n^*(t)=e^{t\triangle}(y(u_n^*)(0)-y_n^*(0)) ;
\end{equation}

\item for every $t\in (\tau_1,\tau_2]$,
\begin{eqnarray}\label{infty-estimate:8}
y(u_n^*)(t)-y_n^*(t)
&=&e^{t\triangle}(y(u_n^*)(0)-y_n^*(0))
+\d{\int_{\tau_1}^t} e^{(t-s)\triangle}\chi_\omega h_n^{-1}u_{1,n}^*\,\mathrm ds
-e^{(t-\tau_1)\triangle}\chi_\omega u_{1,n}^* \\
&=&e^{t\triangle}(y(u_n^*)(0)-y_n^*(0))+\d{\int_{\tau_1}^t} e^{(t-s)\triangle}\chi_\omega p_n^*(\tau_1)\,\mathrm ds -\d{\int_{\tau_1}^{\tau_2}} e^{(t-\tau_1)\triangle}\chi_\omega p_n^*(\tau_1)\,\mathrm ds ; \nonumber
\end{eqnarray}

\item for every $t\in (\tau_{i-1},\tau_i]$, with $i\geq 3$,
\begin{eqnarray}\label{infty-estimate:9}
y(u_n^*)(t)-y_{n}^*(t)
&=&e^{t\triangle}(y(u_n^*)(0)-y_{n}^*(0))
+\d{\sum_{j=2}^{i-1}} \Big(\d{\int_{\tau_{j-1}}^{\tau_j}} e^{(t-s)\triangle}\chi_\omega u_n^*(s)\,\mathrm ds
-e^{(t-\tau_{j-1})\triangle}\chi_\omega u_{j-1,n}^*\Big) \nonumber\\
&\qquad&+\d{\int_{\tau_{i-1}}^t} e^{(t-s)\triangle}\chi_\omega u_n^*(s)\,\mathrm ds
-e^{(t-\tau_{i-1})\triangle}\chi_\omega u_{i-1,n}^*    \\
&=&e^{t\triangle}(y(u_n^*)(0)-y_{n}^*(0))
+\d{\sum_{j=2}^{i-1}} \d{\int_{\tau_{j-1}}^{\tau_j}} e^{(t-s)\triangle}[I-e^{(s-\tau_{j-1})\triangle}]\chi_\omega p_n^*(\tau_{j-1})\,\mathrm ds  \nonumber\\
&\qquad&+\d{\int_{\tau_{i-1}}^t} e^{(t-s)\triangle}\chi_\omega p_n^*(\tau_{i-1})\,\mathrm ds
-\d{\int_{\tau_{i-1}}^{\tau_i}}e^{(t-\tau_{i-1})\triangle}\chi_\omega p_n^*(\tau_{i-1})\,\mathrm ds.\nonumber
\end{eqnarray}
\end{itemize}
It follows from (\ref{Prin-2}), Lemma~\ref{Energy} and (\ref{error-11}) that
\begin{equation}\label{infty-estimate:10}
\|p_n^*\|_{C([0,T];H_0^1(\Omega))}\leq  C(T) \|y_d\|_{L^2(0,T;L^2(\Omega))} ,
\end{equation}
which, combined with (\ref{infty-estimate:7}), (\ref{infty-estimate:8}), (\ref{infty-estimate:9}) and (\ref{infty-estimate:6}), implies that
\begin{eqnarray*}
&&\|y(u_n^*)-y_n^*\|_{L^\infty(0,T;L^2(\Omega))}\\
&\leq&
  C(T)  \displaystyle{\sum_{j=2}^n \int_{\tau_{j-1}}^{\tau_j}} \|[I-e^{(s-\tau_{j-1})\triangle} ]\chi_\omega p^*_n(\tau_{j-1})\| \,\mathrm ds
 +  C(T)  h_n \|y_d\|_{L^2(0,T;L^2(\Omega))} .
\end{eqnarray*}
This, together with Lemma~\ref{intuitive} and (\ref{infty-estimate:10}), yields
\begin{eqnarray*}
&&\|y(u_n^*)-y_n^*\|_{L^\infty(0,T;L^2(\Omega))}\\
&\leq&
  C(T)  \displaystyle{\sum_{j=2}^n \int_{\tau_{j-1}}^{\tau_j}} (s-\tau_{j-1})^{\f{1}{2N}}\|p^*_n(\tau_{j-1})\|_{H_0^1(\Omega)}\,\mathrm ds
 + C(T)  h_n \|y_d\|_{L^2(0,T;L^2(\Omega))} \\
 &\leq&  C(T)  h_n^{1/2N} \|y_d\|_{L^2(0,T;L^2(\Omega))} ,
\end{eqnarray*}
and (\ref{infty-estimate:5}) follows. This ends the proof.

\subsection{Proof of Theorem \ref{thm_SD}}\label{sec_proof_thm_SD}
\subsubsection{Proof of the control error estimate}
Let us establish \eqref{errorSDcontrol}.
As in Section \ref{sec232}, the proof goes in three steps.

\paragraph{Step 1.} We claim that
\begin{equation}\label{add_5}
\sum_{i=1}^n\int_{\tau_{i-1}}^{\tau_i} \|u^*-v_{i,n}^*\|^2\,\mathrm dt = I_1+I_2 ,
\end{equation}
with
$$
I_1 = \sum_
{i=1}^n\int_{\tau_{i-1}}^{\tau_i} \langle \chi_\omega p^*-\chi_\omega \bar{p}_n^*,u^*-v_{i,n}^*\rangle \,\mathrm dt,
$$
and
$$
I_2 = \sum_{i=1}^n\int_{\tau_{i-1}}^{\tau_i} \left\langle \chi_\omega \bar{p}_n^*-
\frac{1}{h_n}\chi_\omega \int_{\tau_{i-1}}^{\tau_i} \bar{p}_n^*(s)\,\mathrm ds, u^*-v_{i,n}^*\right\rangle \,\mathrm dt,
$$
where $p^*$ is defined by \eqref{Maxi-2} and $\bar{p}_n^*$ is defined by \eqref{PMP_SD2}.

The claim follows from \eqref{Maxi-3}, (\ref{PMP_SD3}) and from the fact that
$$
\begin{array}{lll}
\displaystyle{\sum_{i=1}^n\int_{\tau_{i-1}}^{\tau_i}} \|u^*-v_{i,n}^*\|^2\,\mathrm dt
&=&\displaystyle{\sum_{i=1}^n\int_{\tau_{i-1}}^{\tau_i}} \langle u^*-v_{i,n}^*,u^*-v_{i,n}^*\rangle \,\mathrm dt\\
&=&\displaystyle{\sum_{i=1}^n\int_{\tau_{i-1}}^{\tau_i}} \left\langle
\chi_\omega p^*- \frac{1}{h_n}\chi_\omega \int_{\tau_{i-1}}^{\tau_i} \bar{p}_n^*(s)\,\mathrm ds,u^*-v_{i,n}^*\right\rangle \,\mathrm dt.
\end{array}
$$

\paragraph{Step 2.} We claim that
\begin{equation}\label{add_6}
I_1\leq 0.
\end{equation}
Indeed, using \eqref{Maxi-1}, \eqref{Maxi-2}, \eqref{PMP_SD1} and \eqref{PMP_SD2}, we get that
\begin{equation}\label{add_7}
\left\{
\begin{array}{lll}
\partial_t (y^*-\bar{y}_n^*)-\triangle (y^*-\bar{y}_n^*)=\chi_\omega (u^*-f_n^*)&\mbox{in}&\Omega\times
(0,T),\\
y^*-\bar{y}_n^*=0&\mbox{on}&\partial \Omega\times (0,T),\\
(y^*-\bar{y}_n^*)(0)=(y^*-\bar{y}_n^*)(T)&\mbox{in}&\Omega ,
\end{array}\right.
\end{equation}
and
\begin{equation}\label{add_8}
\left\{
\begin{array}{lll}
\partial_t (p^*-\bar{p}_n^*)+\triangle (p^*-\bar{p}_n^*)=y^*-\bar{y}_n^*&\mbox{in}&\Omega\times
(0,T),\\
p^*-\bar{p}_n^*=0&\mbox{on}&\partial \Omega\times (0,T),\\
(p^*-\bar{p}_n^*)(0)=(p^*-\bar{p}_n^*)(T)&\mbox{in}&\Omega.
\end{array}\right.
\end{equation}
Multiplying the first equation of \eqref{add_7} by $p^*-\bar{p}_n^*$ and integrating over $\Omega\times (0,T)$, by \eqref{add_7} and \eqref{add_8}, we obtain that
$$
\int_0^T \langle p^*-\bar{p}_n^*,\chi_\omega (u^*-f_n^*)\rangle \,\mathrm dt=-\int_0^T \|y^*-\bar{y}_n^*\|^2\,\mathrm dt\leq 0,
$$
which, combined with \eqref{def_fn*}, gives \eqref{add_6}.

\paragraph{Step 3.} We claim that
\begin{equation}\label{add_9}
|I_2|\leq C(T)h_n \|y_d\|_{L^2(0,T;L^2(\Omega))} 
\left(\sum_{i=1}^n\int_{\tau_{i-1}}^{\tau_i}\|u^*-v_{i,n}^*\|^2\,\mathrm dt\right)^{\frac{1}{2}}.
\end{equation}
Indeed, on one hand, we first note that
\begin{equation}\label{add_10}
|I_2|\leq \sum_{i=1}^n\int_{\tau_{i-1}}^{\tau_i}\left\|\bar{p}_n^*-\frac{1}{h_n}\int_{\tau_{i-1}}^{\tau_i}\bar{p}_n^*(s)\,\mathrm ds\right\| \, \|u^*-v_{i,n}^*\|\,\mathrm dt.
\end{equation}
It is easy to check that, for every $t\in [\tau_{i-1},\tau_i]$,
\begin{multline}\label{add_11}
\left\|\bar{p}_n^*(t)-\frac{1}{h_n}\int_{\tau_{i-1}}^{\tau_i}\bar{p}_n^*(s)\, \mathrm ds\right\|
=\left\|\displaystyle{\frac{1}{h_n}\int_{\tau_{i-1}}^{\tau_i}}(\bar{p}_n^*(t)-\bar{p}_n^*(s))\, \mathrm ds\right\|
=\frac{1}{h_n}\left\|\displaystyle{\int_{\tau_{i-1}}^{\tau_i}\int_s^t}\partial_\tau \bar{p}_n^*(\tau)\,\mathrm d\tau\,\mathrm ds\right\| \\
\leq \frac{1}{h_n}\int_{\tau_{i-1}}^{\tau_i}\int_{\tau_{i-1}}^{\tau_i}\|\partial_\tau \bar{p}_n^*(\tau)\|\,\mathrm d\tau\,\mathrm ds
=\int_{\tau_{i-1}}^{\tau_i}\|\partial_\tau \bar{p}_n^*(\tau)\|\,\mathrm d\tau
\leq h_n^{1/2}\left(\displaystyle{\int_{\tau_{i-1}}^{\tau_i}}\|\partial_\tau \bar{p}_n^*(\tau)\|^2\,\mathrm d\tau\right)^{1/2}.
\end{multline}
It follows from \eqref{add_10}, \eqref{add_11} and from the H\"{o}lder inequality that
\begin{equation}\label{add_12}
\begin{split}
|I_2|&\leq\displaystyle{\sum_{i=1}^n} h_n^{1/2}\left(\displaystyle{\int_{\tau_{i-1}}^{\tau_i}}\|\partial_t \bar{p}_n^*\|^2\,\mathrm dt\right)^{\frac{1}{2}}
\displaystyle{\int_{\tau_{i-1}}^{\tau_i}}\|u^*-v_{i,n}^*\|\,\mathrm dt\\
&\leq h_n\displaystyle{\sum_{i=1}^n} \left(\displaystyle{\int_{\tau_{i-1}}^{\tau_i}}\|\partial_t \bar{p}_n^*\|^2\,\mathrm dt\right)^{\frac{1}{2}}
\left(\displaystyle{\int_{\tau_{i-1}}^{\tau_i}}\|u^*-v_{i,n}^*\|^2\,\mathrm dt\right)^{\frac{1}{2}}\\
&\leq h_n\left(\displaystyle{\int_0^T}\|\partial_t \bar{p}_n^*\|^2\,\mathrm dt\right)^{\frac{1}{2}}
\left(\displaystyle{\sum_{i=1}^n\int_{\tau_{i-1}}^{\tau_i}}\|u^*-v_{i,n}^*\|^2\,\mathrm dt\right)^{\frac{1}{2}}.
\end{split}
\end{equation}
On the other hand, since $(\bar{y}_n^*,V_n^*)$ is optimal (with $V_n^*=(v_{1,n}^*,\dots,v_{n,n}^*)$), we have
$J(\bar{y}_n^*,f_n^*)\leq J(0,0)$, from which it follows that
\begin{equation}\label{add_13}
\int_0^T \|\bar{y}_n^*-y_d\|^2\,\mathrm dt\leq  \int_0^T \|y_d\|^2\,\mathrm dt ,
\end{equation}
and
\begin{equation}\label{wang-5}
 \int_0^T \|f_n^*\|^2\,\mathrm dt=h_n\sum_{i=1}^n\|v_{i,n}^*\|^2\leq \int_0^T \|y_d\|^2\,\mathrm dt .
\end{equation}
By \eqref{PMP_SD2}, \eqref{add_13} and Lemma \ref{Energy}, we get that
$\int_0^T \|\partial_t \bar{p}_n^*\|^2\,\mathrm dt\leq C(T) \int_0^T \|y_d\|^2\,\mathrm dt$ .
This, combined with \eqref{add_12}, implies \eqref{add_9}.

Finally, \eqref{errorSDcontrol} follows from \eqref{add_5}, \eqref{add_6} and \eqref{add_9}.

\subsubsection{Proof of the state  and cost functional error estimates}
We start with establishing   \eqref{errorSDstate}.
Using \eqref{Maxi-1} and \eqref{PMP_SD1}, we have
\begin{equation}\label{add_15}
\left\{
\begin{array}{lll}
\partial_t (y^*-\bar{y}_n^*)-\triangle (y^*-\bar{y}_n^*)=\chi_\omega (u^*-f_n^*)&\mbox{in}&\Omega\times (0,T),\\
y^*-\bar{y}_n^*=0&\mbox{on}&\partial\Omega\times (0,T),\\
(y^*-\bar{y}_n^*)(0)=(y^*-\bar{y}_n^*)(T)&\mbox{in}&\Omega.
\end{array}\right.
\end{equation}
 (\ref{errorSDstate}) follows from  \eqref{add_15}, Lemma~\ref{Energy} and (\ref{errorSDcontrol}).

Finally, by (\ref{def_J}), (\ref{errorSDcontrol}), (\ref{errorSDstate}), (\ref{add_13}) and (\ref{wang-5}), we obtain (\ref{wang-2}).

\bigskip

\noindent{\bf Acknowledgment.}
The first author acknowledges the support and hospitality of Wuhan University, and the support by FA9550-14-1-0214 of the EOARD-AFOSR.
The second and third authors were supported by the National Natural Science Foundation of China under grant 11371285.


\begin{thebibliography}{12}

\bibitem{Banks-Ito} 
H.T. Banks, K. Ito, 
\textit{Approximation in LQR problems for infinite dimensional systems with unbounded input operators},
J. Math. Systems Estim. Control {\bf 7} (1997), 119--122.

\bibitem{Banks-Kunisch} 
H.T. Banks, K. Kunisch, 
\textit{The linear regulator problem for parabolic systems},
SIAM J. Control Optim. {\bf 22} (1984), 684--698.

\bibitem{Bensoussan} 
A. Bensoussan, J.-L. Lions, 
\textit{Impulse control and quasi-variational inequalities}, 
Bordas, Paris, 1984.

\bibitem{Boyer}
F. Boyer, F. Hubert, J. Le Rousseau,
\textit{Uniform controllability properties for space/time-discretized parabolic equations},
Numer. Math. {\bf 118} (2011), no. 4, 601--661.

\bibitem{Boyer2}
F. Boyer, J. Le Rousseau,
\textit{Carleman estimates for semi-discrete parabolic operators and application to the controllability of semi-linear semi-discrete parabolic equations},
Ann. Inst. H. Poincar\'e Anal. Non Lin\'eaire {\bf 31} (2014), no. 5, 1035--1078.

\bibitem{BourdinTrelat}
L. Bourdin, E. Tr\'elat,
\textit{Optimal sampled-data control, and generalizations on time scales},
Preprint Hal (2015), 42 pages.


\bibitem{ErvedozaValein}
S. Ervedoza, J. Valein,
\textit{On the observability of abstract time-discrete linear parabolic equations},
Rev. Mat. Complut. {\bf 23} (2010), no. 1, 163--190.

\bibitem{ErvedozaZuazua}
S. Ervedoza, E. Zuazua,
\textit{Numerical approximation of exact controls for waves},
Springer Briefs in Mathematics. Springer, New York, 2013.

\bibitem{Evans} 
L. Evans, R. Gariepy, 
\textit{Measure theory and fine properties of functions},
Studies in Advanced Mathematics, CRC Press, Boca Raton, Florida, 1992.

\bibitem{Falk} 
F.S. Falk, 
\textit{Approximation of a class of optimal control problems with order of convergence estimates},
J. Math. Anal. Appl. {\bf 44} (1973), 28--47.

\bibitem{Gilbarg} 
D. Gilbarg, N.S. Trudinger, 
\textit{Elliptic partial differential equations of second order}, 
Springer-Verlag, Berlin, 2001.

\bibitem{Gunzburger} 
M.D. Gunzburger, L. Hou, T. Svobodny, 
\textit{Analysis and finite element approximation of optimal control problems for stationary Navier-Stokes equations with Dirichlet controls}, 
RAIRO Model. Math. Anal. Numer. {\bf 25} (1991), 711--748.

\bibitem{Geveci} 
T. Geveci, 
\textit{On the approximation of the solution of an optimal control problem governed by an elliptic equation}, 
RAIRO Anal. Numer. {\bf 13} (1979), 313--328.

\bibitem{Gibson} 
J.S. Gibson, 
\textit{The Riccati integral solutions for optimal control problems on Hilbert spaces}, 
SIAM J. Control Optim. {\bf 17} (1979), 537--565.

\bibitem{Hou} 
L. Hou, J.C. Turner, 
\textit{Analysis and finite element approximation of an optimal control problem in electrochemistry with current density controls}, 
Numer. Math. {\bf 71} (1995), 289--315.


\bibitem{KappelSalamon} 
F. Kappel, D. Salamon,
\textit{An approximation theorem for the algebraic Riccati equation},
SIAM J. Control Optim. {\bf 28} (1990), no. 5, 1136--1147.

\bibitem{Khapalov}
A. Khapalov,
\textit{Exact controllability of second-order hyperbolic equations with impulse controls},
Appl. Anal. {\bf 63} (1996), no. 3-4, 223--238.

\bibitem{Labbe-Trelat} 
S. Labb\'{e}, E. Tr\'{e}lat, 
\textit{Uniform controllability of semidiscrete approximations of parabolic control systems}, Syst. Control Letters {\bf 55} (2006), 597--609.

\bibitem{Lamberton} 
D. Lamberton, 
\textit{\'{E}quations d'\'{e}volution lin\'{e}aires associ\'{e}es \`{a} des semi-groupes de contractions sur les espaces $L^p$}, 
J. Funct. Anal. {\bf 72} (1987), 252--262.

\bibitem{Lasie} 
I. Lasiecka, 
\textit{Boundary control of parabolic systems: finite-element approximation},
Appl. Math. Optim. {\bf 6} (1980), 31--62.

\bibitem{Lasiecka} 
I. Lasiecka, 
\textit{Ritz-Galerkin approximation of the time optimal boundary control problem for parabolic systems with Dirichlet boundary conditions}, 
SIAM J. Control Optim. {\bf 22} (1984), 477--499.

\bibitem{LasieckaTriggiani1}
I. Lasiecka, R. Triggiani,
\textit{Control theory for partial differential equations: continuous and approximation theories. I. Abstract parabolic systems},
Encyclopedia of Mathematics and its Applications, 74, Cambridge University Press, Cambridge, 2000.

\bibitem{Lebeau}
G. Lebeau, J. Le Rousseau, P. Terpolilli, E. Tr\'elat,
\textit{Geometric control condition for the wave equation with time-dependent domains},
Ongoing work (2015).

\bibitem{LiYong}
X. J. Li, J. M. Yong,
\textit{Optimal control theory for infinite-dimensional systems},
Systems \& Control: Foundations \& Applications, Birkh\"auser Boston, Inc., Boston, MA, 1995.

\bibitem{LM}
J.-L. Lions, E. Magenes,
\textit{Probl\`emes aux limites non homog\`enes et applications},
Vol. 1, Travaux et Recherches Math\'ematiques, No. 17, Dunod, Paris, 1968.

\bibitem{Liu} 
W.B. Liu, H.P. Ma, T. Tang, N.N. Yan, 
\textit{A posteriori error estimates for discontinuous Galerkin time-stepping method for optimal control problems governed by parabolic equations}, 
SIAM J. Numer. Anal. {\bf 42} (2004), 1032--1061.


\bibitem{LiuZheng}
Z. Liu, S. Zheng,
\textit{Semigroups associated with dissipative systems},
Chapman \& Hall/CRC Research Notes in Mathematics, 398, Boca Raton, FL, 1999.


\bibitem{PTZObs1}
Y. Privat, E. Tr\'elat, E. Zuazua,
\textit{Optimal observation of the one-dimensional wave equation},
J. Fourier Anal. Appl. {\bf 19} (2013), no. 3, 514--544.

\bibitem{PTZ_HUM}
Y. Privat, E. Tr\'elat, E. Zuazua,
\textit{Optimal location of controllers for the one-dimensional wave equation},
Ann. Inst. H. Poincar\'e Anal. Non Lin\'eaire {\bf 30} (2013), no. 6, 1097--1126.


\bibitem{PTZ_optparab}
Y. Privat, E. Tr\'elat, E. Zuazua,
\textit{Optimal shape and location of sensors or controllers for parabolic equations with random initial data},
Arch. Ration. Mech. Anal. {\bf 216} (2015), no. 3, 921--981.


\bibitem{RosenWang}
I.G. Rosen, C. Wang,
\textit{On the continuous dependence with respect to sampling of the linear quadratic regulator problem for distributed parameter systems},
SIAM J. Control Optim. {\bf 30} (1992), no. 4, 942--974.


\bibitem{Tiba} 
D. Tiba, F. Tr\"{o}ltzsch, 
\textit{Error estimates for the discretization of state constrained convex control problems},
Numer. Funct. Anal. Optim. {\bf 17} (1996), 1005--1028.


\bibitem{WangWang} 
G. Wang, L. Wang, 
\textit{Error estimates for optimal control problems of the heat equation and with the end-point
state constraint}, 
Internat. J. Num. Anal. Modeling {\bf 9} (2012), 844--875.

\bibitem{Wang} 
G. Wang, X. Yu, 
\textit{Error estimates for an optimal control problem governed by the heat equation with state and control constraints}, 
Internat. J. Num. Anal. Modeling {\bf 7} (2010), 30--65.

\bibitem{Yang} 
T. Yang, 
\textit{Impulse control theory}, 
Lecture Notes in Control and Information Sciences {\bf 272}, Springer-Verlag, Berlin, 2001.

\bibitem{Yong} 
J. Yong, P. Zhang, 
\textit{Necessary conditions of optimal impulse controls for distributed parameter systems},
Bull. Australian Math. Soc. {\bf 45} (1992), 305--326.

\bibitem{Zabczyk}
J. Zabczyk,
\textit{Mathematical control theory: an introduction},
Systems \& Control: Foundations \& Applications, Birkh\"auser Boston, Inc., Boston, MA, 1992.

%
%

\bibitem{Zuazua-4} 
E. Zuazua, 
\textit{Propagation, observation and control of waves approximated by finite difference method}, 
SIAM Rev. {\bf 47} (2005), 197--243.


\end{thebibliography}
\end{document}